\documentclass[12pt]{amsart}

\usepackage{amsmath}
\usepackage{amssymb}
\usepackage[all]{xy}

\numberwithin{equation}{section}

\newtheorem{thm}{Theorem}[section]

\newtheorem{lem}[thm]{Lemma}
\newtheorem{prop}[thm]{Proposition}
\newtheorem{cor}[thm]{Corollary}
\newtheorem{dfn-lem}[thm]{Lemma-Definition}
\theoremstyle{definition}
\newtheorem{rem}[thm]{Remark}
\newtheorem{exam}[thm]{Example}
\newtheorem{exam-nota}[thm]{Example-Notation}
\newtheorem{nota}[thm]{Notation}
\theoremstyle{definition}
\newtheorem{dfn}[thm]{Definition}

\theoremstyle{definition}
\newtheorem{dfn-nota}[thm]{Definition-Notation}

\newcommand{\beqa}{\begin{eqnarray*}}
\newcommand{\eeqa}{\end{eqnarray*}}

\newcommand{\fe}{\mbox{${\mathfrak e}$}}
\newcommand{\fa}{\mbox{${\mathfrak a}$}}
\newcommand{\ft}{\mbox{${\mathfrak t}$}}
\newcommand{\fk}{\mbox{${\mathfrak k}$}}
\newcommand{\fg}{\mbox{${\mathfrak g}$}}

\newcommand{\fl}{\mbox{${\mathfrak l}$}}
\newcommand{\fs}{\mbox{${\mathfrak s}$}}
\newcommand{\fsl}{\mbox{${\fs\fl}$}}

\newcommand{\fh}{\mbox{${\mathfrak h}$}}
\newcommand{\fn}{\mbox{${\mathfrak n}$}}

\newcommand{\fp}{\mbox{${\mathfrak p}$}}

\newcommand{\fb}{\mbox{${\mathfrak b}$}}
\newcommand{\fz}{\mbox{${\mathfrak z}$}}
\newcommand{\fm}{\mbox{${\mathfrak m}$}}

\newcommand{\C}{\mbox{${\mathbb C}$}}

\newcommand{\R}{\mbox{${\mathbb R}$}}

\newcommand{\Ad}{{\rm Ad}}

\newcommand{\fgl}{\mathfrak{gl}}

\newcommand{\xifij}{\xi_{f_{i,j}}}

\newcommand{\dn}{{n\choose 2}}
\newcommand{\dnone}{{n+1\choose 2}}

\newcommand{\ad}{\operatorname{ad}}

\newcommand{\calQ}{\mathcal{Q}}
\newcommand{\nilrad}{\fn_{\calQ}}
\newcommand{\borel}{\fb_{\calQ}}
\newcommand{\orbittower}{X_{\calQ}}
\newcommand{\B}{\mathcal{B}}
\newcommand{\dt}{\frac{d}{dt}|_{t=0}}
\newcommand{\brackets}{\{\cdot,\cdot\}}
\newcommand{\V}{\mathcal{V}}
\newcommand{\hv}{\hat{v}}

\newcommand{\T}{\mathcal{T}}

\newcommand{\hQ}{\hat{Q}}
\title{The Gelfand-Zeitlin integrable system and $K$-orbits on the flag variety }
\author[M. Colarusso]{Mark Colarusso}
\address{Department of Mathematics, Idaho State University, Pocatello, ID, 83209}
\email{colamark@isu.edu}

\author[S. Evens]{Sam Evens}
\address{Department of Mathematics, University of Notre Dame, Notre Dame, 46556}
\email{sevens@nd.edu}
\begin{document}
\maketitle
\begin{center}
{\it To Nolan Wallach, on the occasion of his 70th birthday, with gratitude and admiration. }
\end{center}
\begin{abstract}
In this expository paper, we provide an overview of the Gelfand-Zeiltin integrable system on the Lie algebra of $n\times n$ complex matrices $\fgl(n,\C)$
introduced by Kostant and Wallach in 2006.  We discuss results concerning the geometry of the set of strongly regular elements, which consists of the points where 
Gelfand-Zeitlin flow is Lagrangian.  We use the theory of $K_{n}=GL(n-1,\C)\times GL(1,\C)$-orbits
on the flag variety $\mathcal{B}_{n}$ of $GL(n,\C)$ to describe the strongly regular elements in the nilfiber of the moment map of the system.  We give an overview of the general theory of orbits of a symmetric subgroup of a reductive algebraic group acting on its flag variety, and illustrate how the general theory can be applied 
to understand the specific example of $K_{n}$ and $GL(n,\C)$. 

\end{abstract}


\section{Introduction}
In a series of papers \cite{KW1, KW2}, Kostant and Wallach study the action
of an abelian Lie group $A \cong \C^{\frac{n(n-1)}{2}}$ on $\fg = \fgl(n,\C)$.
The Lie algebra $\fa$ of $A$ is the abelian Lie algebra of Hamiltonian
vector fields of the Gelfand-Zeitlin collection of functions $J_{GZ}:=\{ f_{i,j}:i=1,\dots n, \, j=1,\dots i \}$
(see Section \ref{s:gzint} for precise notation).  The set of functions $J_{GZ}$ is Poisson commutative, and
their restriction to each regular adjoint orbit in $\fg$ forms an integrable system.  For each function in the collection, the corresponding Hamilton vector field on $\fg$ is
complete, and the action of $A$ on $\fg$ is given by integrating the
action of $\fa$.   

Kostant and Wallach consider a Zariski open subset $\fg_{sreg}$ of $\fg$,
which consists of all elements $x\in \fg$ such that the differentials
of the functions $J_{GZ}$ are linearly independent
at $x$.  Elements of $\fg_{sreg}$ are called strongly regular,
and Kostant and Wallach show that $\fg_{sreg}$ is exactly the set
of regular elements $x$ of $\fg$ such that the orbit $A\cdot x$ is
Lagrangian in the adjoint orbit of $x$.   In \cite{Col, Col1}, the first
author determined the $A$-orbits in $\fg_{sreg}$ through explicit
computations.   We denote by $\Phi:\fg \to \C^{\frac{n(n+1)}{2}}$ the
map given by $\Phi(x)=( f_{i,j}(x) )$, and note that in \cite{Col, Col1},
the most subtle and interesting case is the nilfiber $\Phi^{-1}(0)$.

The Gelfand-Zeitlin functions are defined using a sequence of projections
$\pi_i:\fgl(i,\C) \to \fgl(i-1,\C)$ given by mapping an $i\times i$ matrix $y$ to its $(i-1)\times (i-1)$ submatrix in 
the upper left hand corner.  Our paper \cite{CEKorbs} 
exploits the fact that each projection $\pi_i$ is equivariant with respect to the
action of $GL(i-1,\C)$ on $\fgl(i,\C)$ by conjugation, where $GL(i-1,\C)$ is embedded in the top left hand corner 
of $GL(i,\C)$ in the natural way.  
In particular, we use the theory of
$GL(i-1,\C)$-orbits on the flag variety $\B_i$ of $\fgl(i,\C)$ for $i=1,\dots, n$ 
to provide a more conceptual understanding of the $A$-orbits in the
nilfiber.  In addition, we prove that every Borel subalgebra contains
strongly regular elements, and hope to develop these methods in order
to better understand the topology of $\fg_{sreg}$.

In this expository paper, we review results of Kostant, Wallach,
and the first author, and then explain how to use the theory of 
$GL(i-1,\C)$-orbits on $\B_i$ in order to derive the results from
\cite{CEKorbs}.   In Section 2, we recall the basic symplectic and Poisson 
geometry needed to construct the Gelfand-Zeitlin integrable system.  We then discuss
the work of Kostant and Wallach in constructing the system and the action of the group $A$, and 
the work of the first author in describing the $A$-orbit structure of $\fg_{sreg}$.  In Section 3, we give an overview
of our results from \cite{CEKorbs} and sketch some of the proofs.  In Section 4, we review
the rich theory of orbits of a symmetric subgroup $K$ on the flag variety
$\B$ of a reductive group $G$, as developed by Richardson, Springer,
and others.   In particular, we show explicitly how the theory applies
if $K=GL(n-1,\C) \times GL(1, \C)$ and $G=GL(n,\C)$, and we hope this section
will make the general theory of $K$-orbits more accessible to 
researchers interested in applying this theory.

It would be difficult to overstate the influence of Nolan Wallach on
the work discussed in this paper.   We look forward to further stimulating
interactions with Nolan in the future, and note that our plans for
developing this work may well depend on utilizing completely different
work of Nolan than that discussed here.
 The work
by the 
second author was partially supported by NSA grants H98230-08-0023 and
 H98230-11-1-0151.

 

\section{The Gelfand-Zeitlin integrable system on $\fgl(n,\C)$}\label{s:gzint}

\subsection{Integrable Systems}\label{ss:intsys}

In this section, we give a brief discussion of integrable systems.  For further details, we refer the reader to \cite{AVMVH}, \cite{Audin}.  We denote by $M$ an analytic (or smooth) manifold with holomorphic (smooth) functions $\mathcal{H}(M)$.  

Let $(M,\omega)$ be a $2n$-dimensional symplectic manifold with symplectic form $\omega\in \wedge^{2} T^{*}M$.  For $f\in \mathcal{H}(M)$, we let $\xi_{f}$ be the unique vector field such that 
\begin{equation}\label{eq:symHam}
 df(Y)=\omega(Y,\xi_{f}),
\end{equation}
for all vector fields $Y$ on $M$.  The vector field $\xi_{f}$ is called the \emph{Hamiltonian} vector field of $f$.  We can use these vector fields to give $\mathcal{H}(M)$ the structure of a Poisson algebra with Poisson bracket:
\begin{equation}\label{eq:symplecticPois}
\{f, g\}:=\omega(\xi_{f}, \xi_{g}), 
\end{equation}
for $f,\, g\in\mathcal{H}(M)$.  That is to say that $\brackets$ makes $\mathcal{H}(M)$ into a Lie algebra and $\brackets$ satisfies a Leibniz rule with respect to the associative multiplication of $\mathcal{H}(M)$.  

To define an integrable system on $(M,\omega)$, we need the following notion.
\begin{dfn}
We say the functions $\{F_{1},\dots, F_{r}\}\subset \mathcal{H}(M)$ are \emph{independent} if the open set $U=\{m\in M:\, (dF_{1})_{m}\wedge\dots\wedge (dF_{r})_{m}\neq 0\}$  is dense in $M$. 
\end{dfn}  

\begin{dfn}\label{dfn:intsystem}
Let $(M,\omega)$ be a $2n$-dimensional symplectic manifold.  An \emph{integrable system} on $M$ is a collection of $n$ independent functions $\{F_{1},\dots, F_{n}\}\subset\mathcal{H}(M)$ such that $\{F_{i}, F_{j}\}=0$ for all $i, j$.  
\end{dfn}

\begin{rem}\label{r:classical}
This terminology originates in Hamiltonian mechanics.  In that context, $(M,\omega, H)$ is a phase space of a classical Hamiltonian system with $n$ degrees of freedom and Hamiltonian function $H\in\mathcal{H}(M)$ (the total energy of the system).  The trajectory of the Hamiltonian vector field $\xi_{H}$ describes the time evolution of the system.  If we are given an integrable system $\{F_{1}=H,\, \dots, F_{n}\}$, then this trajectory can be found using only the operations of function integration and function inversion  (\cite{AVMVH}, Section 4.2).  Such a Hamiltonian system is said to be integrable by quadratures.  

\end{rem}

Integrable systems are important in Lie theory, because they are useful in 
geometric constructions of representations through the theory of quantization \cite{GS},\cite{E}  (see Remark \ref{r:GS} below).  For example, integrable systems provide a way to construct polarizations of symplectic manifolds $(M,\omega)$.  By a polarization, we mean an integrable subbundle of the tangent bundle $P\subset TM$ such that each of the fibers $P_{m}\subset T_{m}M$, is Lagrangian, i.e. $P_{m}=P_{m}^{\perp}$, where $P_{m}^{\perp}$ is the annihilator of $P_{m}$ with respect to the symplectic form $\omega_{m}$ on $T_{m}M$.  A submanifold $S\subset (M,\omega)$ is said to be Lagrangian if $T_{m}(S)$ is Lagrangian for each $m\in S$, so that the leaves of a polarization are Lagrangian submanifolds of $M$.  The existence of a polarization is a crucial ingredient in constructing a geometric quantization of $M$ (for $M$ a real manifold) (see for example, \cite{Woodhouse}), and Lagrangian submanifolds are also important in the study of deformation quantization (see
for example, \cite{NT}).
 
To see how an integrable system on $(M,\omega)$ gives rise to a polarization, we consider the moment map of the system $\{F_{1},\dots, F_{n}\}$: 
\begin{equation}\label{eq:momentum}
\mathbf{F}:M\to K^{n},\; \mathbf{F}(m)=(F_{1}(m),\dots, F_{n}(m)) \mbox{ for } m\in M,
\end{equation}
where $K=\mathbb{R},\,\C$.  Let $U=\{m\in M: (dF_{1})_{m}\wedge\dots\wedge (dF_{n})_{m}\neq 0\}$ and let $P\subset TU$ be $P=\mbox{span}\{\xi_{F_{i}}:\, i=1,\dots, n\}$.  Then $P$ is a polarization of the symplectic manifold $(U,\omega|_{U})$ whose leaves are the connected components of the level sets of $\mathbf{F}|_{U}$, i.e. the regular level sets of $\mathbf{F}$.  Indeed, if $S\subset U$ is a regular level set of $\mathbf{F}$, then $\dim S=\dim M-n=n$.  It then follows that for $m\in S$, $T_{m}(S)=\{(\xi_{F_{i}})_{m}:\, i=1,\dots , n\},$ since the vector fields $\xi_{F_{1}},\dots, \xi_{F_{n}}$ are tangent to $S$ and independent on $U$.  Thus, $T_{m}(S)$ is isotropic by Equation (\ref{eq:symplecticPois}) and of dimension $\dim S=n=\frac{1}{2}\dim U$, so that $T_{m}(S)$ is Lagrangian.  




\subsection{Poisson manifolds and the Lie-Poisson structure}

To study integrable systems in Lie theory, we need to consider not only symplectic manifolds, but Poisson manifolds.  We briefly review some of the basic elements of Poisson geometry here.  For more detail, we refer the reader to \cite{Va},\cite{AVMVH}.

  A Poisson manifold $(M,\{\cdot, \cdot\})$ is an analytic (smooth) manifold where the functions $\mathcal{H}(M)$ have the structure of a Poisson algebra with Poisson bracket $\{\cdot, \cdot\}$.  For example, any symplectic manifold is a Poisson manifold where the Poisson bracket of functions is given by Equation (\ref{eq:symplecticPois}).  For a Poisson manifold $(M,\brackets)$,  the Hamiltonian vector field for $f\in\mathcal{H}(M)$ is given by 
\begin{equation}\label{eq:Ham1}
\xi_{f}(g):=\{f,g \}, 
\end{equation}
$g\in\mathcal{H}(M)$.  In the case where $(M,\omega)$ is symplectic, it is easy to see that this definition of Hamiltonian vector field agrees with the definition given in Equation (\ref{eq:symHam}). 

If we have two Poisson manifolds $(M_{1},\brackets_{1})$ and $(M_{2},\brackets_{2})$, an analytic (smooth) map $\Phi : M_1 \to M_2$ is said to be \emph{Poisson} if 
  \begin{equation}\label{eq:Poissonmap}
  \{f\circ\Phi, g\circ \Phi\}_{1}=\{f,g\}_{2}\circ\Phi,
  \end{equation}
for $f,\,g \in \mathcal{H}(M_{2})$.  That is to say, $\Phi^{*}:\mathcal{H}(M_{2})\to\mathcal{H}(M_{1})$ is a homomorphism of Poisson algebras. Equivalently, for $f\in \mathcal{H}(M_{2})$,
\begin{equation}\label{eq:restHam}
\Phi_{*} \xi_{\Phi^{*} f}=\xi_{f}.  
\end{equation} 
In particular, a submanifold $(S,\brackets_{S})\subset (M,\brackets_{M})$ with Poisson structure $\brackets_{S}$ is said to be a \emph{Poisson submanifold} of $(M,\brackets_{M})$ if the inclusion $i:S\hookrightarrow M$ is Poisson.

In general, Poisson manifolds $(M,\brackets)$ are not symplectic, but they are foliated by symplectic submanifolds called \emph{symplectic leaves}.  Consider the (singular) distribution on $M$ given by
   \begin{equation}\label{eq:chardistr}
   \chi (M):=\mbox{span}\{ \xi_{f}:\, f\in \mathcal{H}(M)\}.
   \end{equation}
The distribution $\chi(M)$ is called the \emph{characteristic distribution} of $(M,\brackets)$.  Using the Jacobi identity for the Poisson bracket $\brackets$, one computes that
\begin{equation}\label{eq:bracket}
   [\xi_{f},\xi_{g}]=\xi_{\{f,g\}}, 
   \end{equation}
   so that the distribution $\chi(M)$ is involutive.  Using a general version of the Frobenius theorem, one can then show that $\chi(M)$ is integrable and the leaves $(S,\brackets_{S})$ are Poisson submanifolds of $(M,\brackets)$, where the Poisson bracket $\brackets_{S}$ is induced by a symplectic form $\omega_{S}$ on $S$ as in Equation (\ref{eq:symplecticPois}).  For further details, see \cite{Va}, Chapter 2.  

  
  
Let $\fg$ be a reductive Lie algebra over $\R$ or $\C$ and let $G$ be any connected Lie group with Lie algebra $\fg$.  Let $\beta(\cdot, \cdot)$ be a non-degenerate, $G$-invariant bilinear form on $\fg$.  Then $\fg$ has the structure of a Poisson manifold, which we call the Lie-Poisson structure.  If $f\in\mathcal{H}(\fg)$, we can use the form $\beta$ to identify the differential $df_{x}\in T_{x}^{*}(\fg)=\fg^{*}$ at $x\in\fg$ with an element $\nabla f(x)\in\fg$.  The element  $\nabla f(x)$ is determined by its pairing against $z\in \fg \cong T_x(\fg)$ by the formula,  
  \begin{equation}\label{eq:differential}
  \beta(\nabla f(x), z)=\dt f (x+tz)=df_{x}(z).
  \end{equation}
 We then define a Poisson bracket on $\mathcal{H}(\fg)$ by:
   \begin{equation}\label{eq:Pbracket}
  \{f,g\}(x)=\beta(x, [\nabla f(x),\nabla g(x)]). 
  \end{equation}
It can be shown that this definition of the Poisson structure on $\fg$ is independent of the choice of form $\beta$ in the sense that a different form gives rise to an isomorphic Poisson manifold structure on $\fg$.




From (\ref{eq:Pbracket}) it follows that
  \begin{equation}\label{eq:Ham}
  (\xi_{f})_{x}=[x,\nabla f(x)]\in T_{x}(\fg)=\fg.
  \end{equation}
For $x\in\fg$, let $G\cdot x$ denote its adjoint orbit.  From Equation (\ref{eq:Ham}), it follows that the fiber of the characteristic distribution of $(\fg,\brackets)$ at $x$ is 
$$(\chi(\fg))_{x}=\{[x,y]: y\in\fg\}= T_{x}(G\cdot x).$$
 One can then show that the symplectic leaves of $(\fg,\brackets)$ are the adjoint orbits of $G$ on $\fg$ with the canonical Kostant-Kirillov-Souriau (KKS) symplectic structure (see for example, \cite{CG}, Proposition 1.3.21).  Since $G\cdot x\subset \fg$ is a Poisson submanifold, it follows from Equations (\ref{eq:Poissonmap}) and (\ref{eq:restHam}) that 
  \begin{equation}\label{eq:compat}
  \{f,g\}_{LP}|_{G\cdot x}=\{f|_{G\cdot x}, g|_{G\cdot x}\}_{KKS} \mbox { and }  \xi_{f}^{LP}|_{G\cdot x}=\xi^{KKS}_{f|_{G\cdot x}}
  \end{equation}
  for $f,\, g\in\mathcal{H}(\fg)$, where the Poisson bracket and Hamiltonian field on the left side of the equations are defined using the Lie-Poisson structure, and on the right side they are defined using the KKS symplectic structure as in Section \ref{ss:intsys}.    

This description of the symplectic leaves allows us to easily identify the Poisson central functions of $(\fg,\brackets)$.  We call a function $f\in\mathcal{H}(\fg)$ a \emph{Casimir} if $\{f,g\}=0$ for all $g\in\mathcal{H}(\fg)$.  Clearly, $f$ is a Casimir if and only if $\xi_{f}=0$.  Equation (\ref{eq:compat}) implies this occurs if and only if $df|_{G\cdot x}=0$, since $G\cdot x$ is symplectic.  Thus, the Casimirs for the Lie-Poisson structure on $\fg$ are precisely the $\Ad(G)$-invariant functions, $\mathcal{H}(\fg)^{G}$.

 The symplectic leaves of $(\fg,\brackets)$ of maximal dimension play an important role in our discussion.  For $x\in\fg$, let $\fz_{\fg}(x)$ denote the centralizer of $x$.  We call an element $x\in\fg$ \emph{regular} if $\dim \fz_{\fg}(x)=\mbox{rank}(\fg)$ is minimal \cite{K}.  The orbit $G\cdot x$ then has maximum possible dimension, i.e., $\dim (G\cdot x)=\dim\fg-\mbox{rank}(\fg)$.

  \subsection{Construction of the Gelfand-Zeitlin integrable system on $\fgl(n,\C)$}
  
  
 Let $\fg=\fgl(n,\C)$ and let $G=GL(n,\C)$.  Then $\fg$ is reductive with non-degenerate, invariant form $\beta(x,y)=tr(xy)$, where $tr(xy)$ denote the trace of the matrix $xy$ for $x,\, y\in \fg$.  Thus, $\fg$ is a Poisson manifold with the Lie-Poisson structure.  In this section, we construct an independent, Poisson commuting family of functions on $\fg$, whose restriction to each regular adjoint orbit $G\cdot x$ forms an integrable system in sense of Definition \ref{dfn:intsystem}.  We refer to this family of functions as the Gelfand-Zetilin integrable system on $\fg$.  The family is constructed using Casimir functions for certain Lie subalgebras of $\fg$ and extending these functions to Poisson commuting functions on all of $\fg$.  
  
 We consider the following Lie subalgebras of $\fg$.  For $i=1,\dots, n-1$, we embed $\fgl(i,\C)$ into $\fg$ in the upper left corner and denote its image by $\fg_{i}$.  That is to say, $\fg_{i}= \{ x\in\fg: x_{k,j}=0, \mbox{ if } k>i \mbox{ or }  j>i\}$.   Let $G_{i}\subset GL(n,\C)$ be the corresponding closed subgroup.  If $\fg_{i}^{\perp}$ denotes the orthogonal complement of $\fg_{i}$ with respect to the form $\beta$, then $\fg=\fg_{i}\oplus \fg_{i}^{\perp}$.  Thus, the restriction of the form $\beta$ to $\fg_{i}$ is non-degenerate, so we can use it to define the Lie-Poisson structure of $\fg_{i}$ via Equation (\ref{eq:Pbracket}).  We have a natural projection $\pi_{i}:\fg\to\fg_{i}$ given by $\pi_{i}(x)=x_{i}$, where $x_{i}$ is the upper left $i\times i $ corner of $x$, that is, $(x_{i})_{k,j}=x_{k,j}$ for $1\leq k,j\leq i$ and is zero otherwise.  The following lemma is the key ingredient in the construction of the Gelfand-Zeitlin integrable system on $\fg$.  
 
 \begin{lem}\label{l:proj}
 The projection $\pi_{i}: \fg\to\fg_{i}$ is Poisson with respect to the Lie-Poisson structures on $\fg$ and $\fg_{i}$.  
 \end{lem} 
 \begin{proof}
  Since the Poisson brackets on $\mathcal{H}(\fg)$ and $\mathcal{H}(\fg_{i})$ satisfy the Leibniz rule, it suffices to show Equation (\ref{eq:Poissonmap}) for linear functions $\lambda_{x}$, $\mu_{y}\in\mathcal{H}(\fg_{i})$, where $\lambda_{x}(z)=\beta(x,z)$ and $\mu_{y}(z)=\beta(y,z)$ for $x,\,y,\, z\in\fg_{i}$.  This is an easy computation using the definition of the Lie-Poisson structure in Equation (\ref{eq:Pbracket}) and the decomposition $\fg=\fg_{i}\oplus \fg_{i}^{\perp}$.


 \end{proof}
 
 Let $\C[\fg]$ denote the algebra of polynomial functions on $\fg$.  
 Let 
 \begin{equation}\label{eq:tensor}
 J(n):=< \pi_{1}^{*} (\C[\fg]^{G_{1}}),\dots, \pi_{n-1}^{*}(\C[\fg_{n-1}]^{G_{n-1}}), \C[\fg]^{G}>
 \end{equation}
 be the associative subalgebra of $\C[\fg]$ generated by $\pi_{i}^{*}(\C[\fg_{i}]^{G_{i}})$ for $i\leq n-1$ and $\C[\fg]^{G}$.  
 
 \begin{prop}\label{p:GZalg}
 The algebra $J(n)$ is a Poisson commutative subalgebra of $\C[\fg]$.  
  \end{prop}
 \begin{proof}
  The proof proceeds by induction on $n$, the case $n=1$ being trivial.  Suppose that $J(n-1)$ is Poisson commutative.  Then $J(n)=<\pi_{n-1}^{*}(J(n-1)), \C[\fg]^{G}>$ is the associative algebra generated by $\pi_{n-1}^{*}(J(n-1))$ and $ \C[\fg]^{G}$.  By Lemma \ref{l:proj}, $\pi_{n-1}^{*} (J(n-1))$ is Poisson commutative, and the elements of $\C[\fg]^{G}$ are Casimirs, so that $J(n)$ is Poisson commutative. 
  \end{proof}
  
  \begin{rem}
  It can be shown that the algebra $J(n)$ is a maximal Poisson commutative subalgebra of $\C[\fg]$ (\cite{KW1}, Theorem 3.25).
  \end{rem}
  
  The Gelfand-Zeitlin integrable system is obtained by choosing a set of generators for the algebra $J(n)$.  Let $\C[\fg_{i}]^{G_{i}}=\C[f_{i,1},\dots, f_{i,i}]$, where $f_{i,j}=tr(x_{i}^{j})$ for $j=1,\dots, i$.  Then the functions 
  \begin{equation}\label{eq:GZfun}
  J_{GZ}:= \{ f_{i,j}: i=1,\dots, n,\, j=1,\dots, i\}
  \end{equation}
  generate the algebra $J(n)$ as an associative algebra.  We claim that $J_{GZ}$ is an independent, Poisson commuting set of functions whose restriction to each regular $G\cdot x$ forms an integrable system.  
  
  By Proposition \ref{p:GZalg}, the functions $J_{GZ}$  Poisson commute.  To see that the functions $J_{GZ}$ are independent,  we study the following morphisms:
  $$
  \Phi_{i}:\fg_{i}\to\C^{i},\; \Phi_{i}(y)=(f_{i,1}(y),\dots, f_{i,i}(y)),
  $$
  for $i=1,\dots, n$.  We define the Kostant-Wallach map to be the morphism
  \begin{equation}\label{eq:KWmap}
  \Phi: \fg\to\C^{\dnone} \mbox{ given by } \Phi(x)=(\Phi_{1}(x_{1}),\dots, \Phi_{i}(x_{i}),\dots, \Phi_{n}(x_{n})).
  \end{equation}
For $z\in \fg_i$, let $\sigma_i(z)$ equal the collection of
$i$ eigenvalues of $z$ counted with repetitions,
 where here we regard $z$ as an $i\times i$ matrix.

\begin{rem}\label{rem_git}
If $x, y \in \fg$, then $\Phi(x)=\Phi(y)$ if and only if
$\sigma_i(x_i)=\sigma_i(y_i)$ for $i=1, \dots, n$.  This follows from the fact that $\C[\fg_{i}]^{G_{i}}=\C[f_{i,1},\dots, f_{i,i}]=\C[p_{i,1},\dots, p_{i,i}]$, where $p_{i,j}$ is the coefficient of $t^{j-1}$ in the characteristic polynomial of $x_{i}$ thought of as an $i\times i$ matrix.  In particular, $\Phi(x)=(0,\dots, 0)$ if and only if $x_{i}$ is nilpotent for $i=1,\dots, n$. 
\end{rem}

Kostant and Wallach produce a cross-section to the map $\Phi$ using the (upper) Hessenberg matrices.  For $1\leq i,j\leq n$, let $E_{i,j}\in\fg$ denote the elementary matrix with $1$ in the $(i,j)$-th entry and zero elsewhere.  Let $\fb_{+}\subset\fg$ be the standard Borel subalgebra of upper triangular matrices and let $e=\sum_{i=2}^{n} E_{i,i-1}$ be the sum of the negative simple root vectors.  We call elements of the affine variety $e+\fb$ (upper) Hessenberg matrices:
$$
e+\fb=\left [\begin{array}{ccccc}
a_{11} & a_{12} &\cdots & a_{1n-1} & a_{1n}\\
1 & a_{22} &\cdots & a_{2n-1} & a_{2n}\\
0 & 1 & \cdots & a_{3n-1}& a_{3n}\\
\vdots &\vdots &\ddots &\vdots &\vdots\\
0 & 0 &\cdots & 1 &a_{nn}\end{array}\right ]_{\mbox{\large .}}
$$
Kostant and Wallach prove the following remarkable fact (\cite{KW1}, Theorem 2.3).  
\begin{thm}\label{thm:Hess}
The restriction of the Kostant-Wallach map $\Phi|_{e+\fb}: e+\fb\to \C^{\dnone}$ to the Hessenberg matrices $e+\fb$ is an isomorphism of algebraic varieties.  
\end{thm} 

\begin{rem}\label{r:Ritz}
For $x\in\fg$, let $\mathcal{R}(x)=\{\sigma_{1}(x_{1}),\dots, \sigma_{i}(x_{i}), \dots, \sigma_{n}(x)\}$ be the collection of $\dnone$-eigenvalues of $x_{1},\dots, x_{i}, \dots, x$ counted with repetitions.  The numbers $\mathcal{R}(x)$ are called the Ritz values of $x$ and play an important role in numerical linear algbera (see for example \cite{Parbook},\cite{PS}).  In this language, Theorem \ref{thm:Hess} says that any $\dnone$-tuple of complex numbers can be the Ritz values of an $x\in\fg$ and that there is a unique Hessenberg matrix having those numbers as Ritz values.  Contrast this with the Hermitian case in which the necessarily real eigenvalues of $x_{i}$ must interlace those of $x_{i-1}$ (see for example \cite{HJ}).  This discovery has led to some new work on Ritz values by linear algebaists \cite{PS},\cite{PShom}.  

\end{rem}


Theorem \ref{thm:Hess} suggests the following definition from \cite{KW1}.  
\begin{dfn}\label{d:sreg}
We say that $x\in\fg$ is strongly regular if the differentials $\{(df_{i,j})_{x}: i=1,\dots n,\, j=1,\dots, i\}$ are linearly independent.  We denote the set of strongly regular elements of $\fg$ by $\fg_{sreg}$.  
\end{dfn}
  By Theorem \ref{thm:Hess}, $e+\fb\subset\fg_{sreg}$, and since $\fg_{sreg}$ is Zariski open, it is  dense in both the Zariski topology and the Hausdorff topology on $\fg$ \cite{M}.  Hence, the polynomials $J_{GZ}$ in (\ref{eq:GZfun}) are independent.  For $c\in\C^{\dnone}$, let $\Phi^{-1}(c)_{sreg}:=\Phi^{-1}(c)\cap \fg_{sreg}$ denote the strongly regular elements of the fiber $\Phi^{-1}(c)$.  It follows from Theorem \ref{thm:Hess} that $\Phi^{-1}(c)_{sreg}$ is nonempty for any $c\in\C^{\dnone}$.  

By a well-known result of Kostant \cite{K}, if $x$ is strongly regular, then $x_{i}\in\fg_{i}$ is regular for all $i$.  We state several equivalent characterizations of strong regularity.
\begin{prop}\label{p:sreg}(\cite{KW1}, Proposition 2.7 and Theorem 2.14)
The following statements are equivalent. 
\begin{enumerate}
\item $x$ is strongly regular.

\item The tangent vectors $\{(\xifij)_{x};\, i=1,\dots, n-1,\, j=1,\dots, i\}$ are linearly independent.

\item The elements $x_{i}\in\fg_{i}$ are regular for all $i=1,\dots, n$ and $\fz_{\fg_{i}}(x_{i})\cap\fz_{\fg_{i+1}}(x_{i+1})=0$ for $i=1,\dots, n-1$, where $\fz_{\fg_{i}}(x_{i})$ denotes the centralizer of $x_{i}$ in $\fg_{i}$. 

\end{enumerate}

\end{prop}

To see that the restriction of the functions $J_{GZ}$ to a regular adjoint orbit $G\cdot x$ form an integrable system, we first observe that $G\cdot x\cap\fg_{sreg}\neq\emptyset$ for any regular $x$.  This follows from the fact that any regular matrix is conjugate to a companion matrix, which is Hessenberg and therefore strongly regular.  Note that the functions $f_{n,1},\dots, f_{n,n}$ restrict to constant functions on $G\cdot x$, so we only consider the restrictions of $\{f_{i,j}: \, i=1,\dots, n-1, \, j=1,\dots, i\}$.  Let $q_{i,j}=f_{i,j}|_{G\cdot x}$ for $i=1,\dots, n-1$ , $j=1,\dots, i$ and let $U=G\cdot x\cap\fg_{sreg}$.  Then $U$ is open and dense in $G\cdot x$.  By Equation (\ref{eq:compat}), Part (2) of Proposition \ref{p:sreg} and Proposition \ref{p:GZalg} imply respectively that the functions $\{q_{i,j}:\, i=1,\dots, n-1, \, j=1,\dots i\}$ are independent and Poisson commute on $U$.  Observe that there are 
$$\sum_{i=1}^{n-1} i=\frac{n(n-1)}{2}=\frac{\dim(G\cdot x)}{2}$$
such functions.  Hence, they form an integrable system on regular $G\cdot x$.



   

  It follows from our work in Section \ref{ss:intsys} that the connected components of the regular level sets of the moment map $y\to  (q_{1,1}(y),\dots, q_{i,j}(y), \dots, q_{n-1,n-1}(y))$ are the leaves of a polarization of $G\cdot x\cap\fg_{sreg}$.  It is easy to see that such regular level sets coincide with certain strongly regular fibers of the Kostant-Wallach map, namely the fibers $\Phi^{-1}(c)_{sreg}$ where $c=(c_{1},\dots, c_{n})$, $c_{i}\in\C^{i}$ with $c_{n}=\Phi_{n}(x)$ (see Equation (\ref{eq:KWmap})).  This follows from Proposition \ref{p:sreg} and the fact that regular matrices which have the same characteristic polynomial are conjugate (see Remark \ref{rem_git}).  
  
  
  We therefore turn our attention to studying the geometry of the strongly regular set $\fg_{sreg}$ and Lagrangian submanifolds $\Phi^{-1}(c)_{sreg}$ of regular $G\cdot x$.  

\begin{rem}\label{r:GS}
The Gelfand-Zeitlin system described here can be viewed as a complexification of the one introduced by Guillemin and Sternberg \cite{GS} on the dual to the Lie algebra of the unitary group.  They show that the Gelfand-Zeitlin integrable system on $\mathfrak{u}(n)^{*}$ is a geometric version of the classical Gelfand-Zeitlin basis for irreducible representations of $U(n)$, \cite{GZ1}.  More precisely, they construct a geometric quantization of a regular, integral coadjoint orbit of $U(n)$ on $\mathfrak{u}(n)^{*}$ using the polarization from the Gelfand-Zeitlin integrable system and show that the resulting quantization is isomorphic to the corresponding highest weight module for $U(n)$ using the Gelfand-Zeitlin basis for the module.   


There is strong empirical evidence (see \cite{PFut}) that the quantum version of the complexified Gelfand-Zeitlin system is the category of  Gelfand-Zeitlin modules studied by Drozd, Futorny, and Ovsienko \cite{DFO}.  These are Harish-Chandra modules for the pair $(U(\fg),\Gamma)$, where $\Gamma\subset U(\fg)$ is the Gelfand-Zeitlin subalgbera of the universal enveloping algbera $U(\fg)$ \cite{FO2}.  It would be interesting to produce such modules geometrically using the geometry of the complex Gelfand-Zeitlin system developed below and deformation quantization.
\end{rem}


\subsection{Integration of the Gelfand-Zeitlin system and the group $A$} 

We can study the Gelfand-Zeitlin integrable system on $\fgl(n,\C)$ and the structure of the fibers $\Phi^{-1}(c)_{sreg}$ by integrating the corresponding Hamiltonian vector fields to a holomorphic action of $\C^{\dn}$ on $\fg$.  The first step is the following observation. 

\begin{thm}\label{thm:complete}
Let $f_{i,j}=tr(x_{i}^{j})$ for $i=1,\dots, n-1$, $j=1,\dots, i$.  Then the Hamiltonian vector field $\xi_{f_{i,j}}$ is complete on $\fg$ and integrates to a holomorphic action of $\C$ on $\fg$ whose orbits are given by: 
\begin{equation}\label{eq:action}
t_{i,j}\cdot x=\Ad(\exp(t_{i,j} jx_{i}^{j-1}))\cdot x,
\end{equation}
for $x\in\fg$, $t_{i,j}\in \C$.  
\end{thm}

\begin{proof} 
Denote the right side of Equation (\ref{eq:action}) by $\theta(t_{i,j}, x)$.  We show that $\theta^{\prime}(t_{i,j}, x)=(-\xifij)_{\theta(t_{i,j}, x)}$ for any $t_{i,j}\in \C$, so that $\theta(-t_{i,j}, x)$ is an integral curve of the vector field $\xifij$.  For the purposes of this computation, replace the variable $t_{i,j}$ by the variable $t$.  Then 
\begin{equation*}
\begin{split}
\frac{d}{dt} |_{t=t_{0}} \,\Ad(
\exp(t\, jx_{i}^{j-1}))\cdot x  &=\ad(jx_{i}^{j-1})\cdot \Ad(
\exp(t_{0}\, jx_{i}^{j-1}))\cdot x\\
&=\ad(j x_{i}^{j-1})\cdot \theta(t_{0}, x).
\end{split}
\end{equation*}
Clearly, $\exp(t_{0}jx_{i}^{j-1})$ centralizes $x_{i}$, so that $\theta(t_{0},\,x)_{i}=x_{i}$.  This implies 
$$
\ad(jx_{i}^{j-1})\cdot\theta(t_{0},x)=\ad(j(\theta(t_{0}, x)_{i})^{j-1})\cdot \theta(t_{0}, x).
$$  
Now it is easily computed that $\nabla f_{i,j}(y)=jy_{i}^{j-1}$ for any $y\in\fg$.  Thus, Equation (\ref{eq:Ham}) implies that 
$$
\ad(j(\theta(t_{0}, x)_{i})^{j-1})\cdot \theta(t_{0}, x)=-(\xifij)_{\theta(t_{0},x)}.
$$

\end{proof} 

We now consider the Lie algbera of Gelfand-Zeitlin vector fields 
\begin{equation}\label{eq:GZvecfields}
\fa:=\mbox{span}\{ \xi_{f_{i,j}}: i=1,\dots, n-1,\, j=1,\dots, i\}.
\end{equation}

By Equation (\ref{eq:bracket}), $\fa$ is an abelian Lie algebra, and since $\fg_{sreg}$ is non-empty, $\dim\fa=\dn$, by (2) of Proposition \ref{p:sreg}.  Let $A$ be the corresponding simply connected Lie group, so that $A\cong \C^{\dn}$.  We take as coordinates on $A$, $$\underline{t}=(\underline{t}_{1},\dots, \underline{t}_{i},\dots, \underline{t}_{n-1})\in\C\times \dots\times\C^{i}\times\dots\times \C^{n-1}=\C^{{n\choose 2}},$$ where $\underline{t}_{i}\in\C^{i}$ with $\underline{t}_{i}=(t_{i1},\dots, t_{ii})$, with $t_{ij}\in \C$ for $i=1,\dots, n-1$, $j=1,\dots, i$.  Since $\fa$ is abelian the actions of the various $t_{i,j}$ given in Equation (\ref{eq:action}) commute.  Thus, we can define an action of $A$ on $\fg$ by composing the actions of the various $t_{i,j}$ in any order.  Thus, for $a=(\underline{t}_{1},\dots, \underline{t}_{n-1})\in A$, $a\cdot x$ is given by the formula:
\begin{equation}\label{eq:Aaction}
a\cdot x=\Ad(\exp(t_{1,1}))\cdot\ldots\cdot\Ad(\exp(jt_{i,j}x_{i}^{j-1}))\cdot\ldots\cdot\Ad(\exp((n-1)t_{n-1,n-1} x_{n-1}^{n-2}))\cdot x.\end{equation}
Theorem \ref{thm:complete} shows that this action integrates the action of $\fa$ on $\fg$, so that
\begin{equation}\label{eq:Atan}
T_{x}(A\cdot x)=\mbox{span}\{( \xifij)_{x}: i=1,\dots, n-1, \,j=1,\dots, i\}.
\end{equation}
Since the functions $J_{GZ}$ Poisson commute, it follows from Equation (\ref{eq:compat}) that $A\cdot x\subset G\cdot x$ is isotropic with respect to the KKS symplectic structure on $G\cdot x$.  Note also that Equation (\ref{eq:Ham1}) implies that $\xi_{f_{i,j}} f_{k,l}=0$ for any $i,\, j$ and $k,\, l$.  It follows that $f_{k,l}$ is invariant under the flow of $\xi_{f_{i,j}}$ for any $i,j$ and therefore is invariant under the action of $A$ given in Equation (\ref{eq:Aaction}).  Thus, the action of $A$ preserves the fibers of the Kostant-Wallach map $\Phi$ defined in Equation (\ref{eq:KWmap}). 

 It follows from Equation (\ref{eq:Atan}) and Part (2) of Proposition \ref{p:sreg} that $x\in\fg_{sreg}$ if and only if $\dim(A\cdot x)=\dn$, which holds if and only if $A\cdot x\subset G\cdot x$ is Lagrangian in regular $G\cdot x$.  Thus, the group $A$ acts on the strongly regular fibers $\Phi^{-1}(c)_{sreg}$ and its orbits form the connected components of the Lagrangian submanifold $\Phi^{-1}(c)_{sreg}\subset G\cdot x$.  Moreover, there are only finitely many $A$-orbits in $\Phi^{-1}(c)_{sreg}$. 


\begin{thm}\label{thm:irred} (\cite{KW1}, Theorem 3.12)
Let $c\in\C^{\dnone}$ and let $\Phi^{-1}(c)_{sreg}$ be a strongly regular fiber of the Kostant-Wallach map.  Then $\Phi^{-1}(c)_{sreg}$ is a smooth algebraic variety of dimension $\dn$ whose irreducible components in the Zariski topology coincide with the orbits of $A$ on $\Phi^{-1}(c)_{sreg}.$  
\end{thm}


Theorem \ref{thm:irred} says that the leaves of the polarization of a regular
adjoint orbit $G\cdot x$ constructed from the Gelfand-Zeitlin integrable system are exactly the $A$-orbits on $G\cdot x\cap\fg_{sreg}$.  

\begin{rem}
Our definition of the Gelfand-Zeitlin integrable system involved choosing the specific set of algebraically independent generators $J_{GZ}$ for the algebra $J(n)$ in Equation (\ref{eq:tensor}).  However, it can be shown that if we choose another algebraically independent set of generators, $J_{GZ}^{\prime}$, then their restriction to each regular adjoint orbit $G\cdot x$ forms an integrable system, and the corresponding Hamiltonian vector fields are complete and integrate to an action of a holomorphic Lie group $A^{\prime}$ whose orbits coincide with those of $A$, \cite{KW1}, Theorem 3.5.  Our particular choice of generators $J_{GZ}$ is to facilitate the easy integration of the Hamiltonian vector fields $\xi_{f}$, $f\in J_{GZ}$ in Theorem \ref{thm:complete}.  

\end{rem}


\subsection{Analysis of the $A$-action on $\Phi^{-1}(c)_{sreg}$}


Kostant and Wallach \cite{KW1} studied the action of $A$ on a special set of regular semisimple elements in $\fg$ defined by: 
\begin{equation}\label{eq:omega}
\fg_{\Omega}=\{x\in\fg: x_{i} \mbox{ is regular semisimple and } \sigma_{i}(x_{i})\cap\sigma_{i+1}(x_{i+1})=\emptyset \mbox{ for all } i\}. 
\end{equation}
Let $\Omega=\Phi(\fg_{\Omega})\subset\C^{\dnone}$.  By Remark \ref{rem_git}, we have $\fg_{\Omega}=\Phi^{-1}(\Omega)$.  In \cite{KW1}, the authors show that the action of $A$ is transitive on the fibers $\Phi^{-1}(c)$ for $c\in\Omega$ and that these fibers are $\dn$-dimensional tori. 
\begin{thm}\label{thm:omega}(\cite{KW1}, Theorems 3.23 and 3.28) 
The elements of $\fg_{\Omega}$ are strongly regular, so that $\Phi^{-1}(c)=\Phi^{-1}(c)_{sreg}$ for $c\in\Omega$.  Moreover, $\Phi^{-1}(c)$ is a homogenuous space for a free, algebraic action of the torus $(\C^{\times})^{\dn}$ and therefore is precisely one $A$-orbit.  
\end{thm}

\begin{rem}
An analogous Gelfand-Zeitlin integrable system exists for complex orthogonal Lie algberas $\mathfrak{so}(n,\C)$.  One can also show that this system integrates to a holomorphic action of $\C^{d}$ on $\mathfrak{so}(n,\C)$, where $d$ is half the dimension of a regular adjoint orbit in $\mathfrak{so}(n,\C)$.  One can then prove the analogue of Theorem \ref{thm:omega} for $\mathfrak{so}(n,\C)$.  We refer the reader to \cite{Col2} for details.     
\end{rem}

The thesis of the first author generalizes Theorem \ref{thm:omega} to an arbitrary fiber $\Phi^{-1}(c)_{sreg}$ for $c\in\C^{\dnone}$ (see \cite{Col}).  The methods used differ from those used to prove Theorem \ref{thm:omega}, but the idea originates in some unpublished work of Wallach, who used a similar strategy to describe the $A$-orbit structure of the set $\fg_{\Omega}$.  We briefly outline this strategy, which can be found in detail in \cite{Col1}, Section 4.  The key observation is that the vector field $\xifij$ acts via Equation (\ref{eq:action}) by the centralizer of $x_{i}$ in $G_{i}$, $Z_{G_{i}}(x_{i})$.  The problem is that the group $Z_{G_{i}}(x_{i})$ is difficult to describe for arbitrary $x_{i}$, so that the formula for the $A$-action in Equation (\ref{eq:Aaction}) is too difficult to use directly.  However, if $x\in\fg_{sreg}$ and $J_{i}$ is the Jordan canonical form of $x_{i}$, then the group $Z_{i}:=Z_{G_{i}}(J_{i})$ is easy to describe, since $x_{i}\in\fg_{i}$ is regular for $i=1,\dots, n$ by (3) of Proposition \ref{p:sreg}.  Further, for $x\in\Phi^{-1}(c)_{sreg}$, $x_{i}$ is in a fixed regular conjugacy class for $i=1,\dots, n$.  This allows us to construct morphisms, $\Phi^{-1}(c)_{sreg}\to G_{i}, \; x\to g_{i}(x)$, where $\Ad(g_{i}(x)^{-1})\cdot x=J_{i}$, where $J_{i}$ is a fixed Jordan matrix (depending only on $\Phi^{-1}(c)_{sreg}$).  We can then use these morphisms to define a free algebraic action of the group $Z:=Z_{1}\times \dots\times Z_{n-1}$ on $\Phi^{-1}(c)_{sreg}$ such that
the $Z$-orbits coincide with the  $A$-orbits.  The action of $Z$ is given by: 
\begin{equation}\label{eq:newact}
(z_{1},\dots, z_{n-1})\cdot x=\Ad(g_{1}(x)z_{1}g_{1}(x)^{-1})\cdot\ldots\cdot \Ad(g_{i}(x)z_{i}g_{i}(x)^{-1})\cdot\ldots\cdot \Ad(g_{n-1}(x)z_{n-1}g_{n-1}(x)^{-1})\cdot x,
\end{equation}
where $z_{i}\in Z_{i}$ for $i=1,\dots, n-1$ and $x\in\Phi^{-1}(c)_{sreg}$, (cf. Equation (\ref{eq:Aaction})).  

The action of the group $Z$ in Equation (\ref{eq:newact}) is much easier to work with than the action of $A$ in Equation (\ref{eq:Aaction}) and allows us to understand the structure of an arbitrary fiber $\Phi^{-1}(c)_{sreg}$.  The first observation is that we can enlarge the set of elements on which the action of $A$ is transitive on the fibers of the Kostant-Wallach map from the set $\fg_{\Omega}$ to the set $\fg_{\Theta}$ defined by:  
$$
\fg_{\Theta}=\{x\in\fg: \sigma_{i}(x_{i})\cap\sigma_{i+1}(x_{i+1})=\emptyset\}.
$$
Let $\Theta=\Phi(\fg_{\Theta})$.  Note that by Remark \ref{rem_git}, $\Phi^{-1}(\Theta)=\fg_{\Theta}$.  

\begin{thm}\label{thm:theta}(\cite{Col1}, Theorem 5.15)
The elements of $\fg_{\Theta}$ are strongly regular.   If $c\in\Theta$, then $\Phi^{-1}(c)=\Phi^{-1}(c)_{sreg}$ is a homogenous space for a free algebraic action of the group $Z=Z_{1}\times \dots \times Z_{n-1}$ given in Equation (\ref{eq:newact}), and thus is exactly one $A$-orbit.  Moreover, $\fg_{\Theta}$ is the maximal subset of $\fg$ for which the action of $A$ is transitive on the fibers of $\Phi$.  
\end{thm}

For general fibers the situation becomes more complicated.  

\begin{thm}\label{thm:general}(\cite{Col1}, Theorem 5.11) 
Let $x\in\fg_{sreg}$ be such that there are $j_{i}$ eigenvalues in common between $x_{i}$ and $x_{i+1}$ for $1\leq i\leq n-1$, and let $c=\Phi(x)$.  Then there are exactly $2^{j}$ $A$-orbits in $\Phi^{-1}(c)_{sreg}$, where $j=\sum_{i=1}^{n-1} j_{i}$.  The orbits of $A$ on $\Phi^{-1}(c)_{sreg}$ coincide with the orbits of a free algebraic action of the group $Z=Z_{1}\times\dots\times Z_{n-1}$ defined on $\Phi^{-1}(c)_{sreg}$ in Equation (\ref{eq:newact}). 
\end{thm}

\begin{rem}\label{rem:bp}
After the proof of Theorem \ref{thm:general} was estabilshed in \cite{Col}, a similar result appeared in an interesting paper of Bielwaski and Pidstrygach \cite{BP}.  Their arguments are independent and completely different from ours.  It would be interesting to study the relation between the two different approaches to establishing the result of Theorem \ref{thm:general}. 

\end{rem}

We highlight a special case of Theorem \ref{thm:general}, which we will investigate in much greater detail below in Section \ref{s:Korbs}.  
\begin{cor}\label{c:nilp}
Consider the strongly regular nilfiber $\Phi^{-1}(0)_{sreg}:=\Phi^{-1}(0,\dots, 0)_{sreg}$.  Then there are exactly $2^{n-1}$ $A$-orbits in $\Phi^{-1}(0)_{sreg}$.  These orbits coincide with the orbits of a free algebraic action of $(\C^{\times})^{n-1}\times\C^{\dn-n+1}$ on $\Phi^{-1}(0)_{sreg}$.
\end{cor}

\begin{proof}
The first statement follows immediately from Remark \ref{rem_git} and Theorem \ref{thm:general}.  For the second statement, we observe that in this case the group $Z=Z_{G_{1}}(e_{1})\times\dots\times Z_{G_{n-1}}(e_{n-1})$, where $e_{i}\in\fg_{i}$ is the principal nilpotent Jordan matrix.  It follows that $Z=(\C^{\times})^{n-1}\times\C^{\dn-n+1}$. 
\end{proof}


Theorem \ref{thm:general} gives a complete description of the local structure of the Lagrangian foliation of regular adjoint orbits of $\fg$ by the Gelfand-Zeitlin integrable system and shows the system is locally algebraically integrable, giving natural algebraic ``angle coordinates" coming from the action of the group $Z=Z_{1}\times\dots\times Z_{n-1}$.  However, Theorem \ref{thm:general} does not say anything about the global nature of the foliation.  Motivated by Theorem \ref{thm:general}, we would like to extend the local $Z$-action on $\Phi^{-1}(c)_{sreg}$ given in (\ref{eq:newact}) to larger subvarieties of $\fg$.  However, this is not possible, except in certain special cases.  The definition of the $Z$-action uses the fact that the Jordan form of each $x_{i}$ for $i=1,\dots, n-1$ is fixed on the fiber $\Phi^{-1}(c)_{sreg}$.  The problem with trying to extend this action is that there is in general no morphism on a larger variety which assigns to $x_{i}$ its Jordan form.  The issue is that the ordered eigenvalues of a matrix are not in general algebraic functions of the matrix entries.

  For the set $\fg_{\Omega}$, Kostant and Wallach resolve this issue by producing an \'{e}tale covering $\fg_{\Omega}(\fe)$ of $\fg_{\Omega}$ on which the eigenvalues are algebraic functions \cite{KW2}.  They then lift the Lie algebra $\fa$ of Gelfand-Zeitlin vector fields in Equation (\ref{eq:GZvecfields}) to the covering where they intergrate to an algebraic action of the torus $(\C^{\times})^{\dn}$.  In our paper \cite{CE}, we extend this to the full strongly regular set using the theory of decomposition classes \cite{BK} and Poisson reduction \cite{EL}.



\section{The geometry of the strongly regular nilfiber}\label{s:Korbs}

In recent work \cite{CEKorbs}, we take a very different approach to describing the geometry of $\fg_{sreg}$ by studying the Borel subalgebras that contain elements of $\fg_{sreg}$.  We develop a new connection between the orbits of certain symmetric subgroups $K_{i}$ on the flag varieties of $\fg_{i}$ for $i=2,\,\dots,\, n$ and the Gelfand-Zeitlin integrable system on $\fg$.  We use this connection to prove that every Borel subalgebra of $\fg$ contains strongly regular elements, and we determine explicitly the Borel subalgebras which contain elements of the strongly regular nilfiber $\Phi^{-1}(0)_{sreg}=\Phi^{-1}(0,\dots, 0)_{sreg}$.  We show that there are $2^{n-1}$ such Borel subalgebras, and that the subvarieties of regular nilpotent elements of these Borel subalgebras are the $2^{n-1}$ irreducible components of $\Phi^{-1}(0)_{sreg}$ given in Corollary \ref{c:nilp}.  This description of the nilfiber is much more explicit than the one given in Corollary \ref{c:nilp}, since the $Z=(\C^{\times})^{n-1}\times\C^{\dn-n+1}$-action of Equation (\ref{eq:newact}) is not easy to compute explicitly.   We refer the reader to
our paper \cite{CEKorbs} for proofs of the results of this section.



\subsection{$K$-orbits and $\Phi^{-1}(0)_{sreg}$}

We begin by considering the strongly regular nilfiber of the Kostant-Wallach map $\Phi^{-1}(0)_{sreg}$.  By Remark \ref{rem_git} and (3) of Proposition \ref{p:sreg}, we note that $x\in\Phi^{-1}(0)_{sreg}$ if and only if the following two conditions are satisfied for every $i=2,\dots, n$:
  \begin{equation}\label{eq:theconditions}
\begin{split}
&(1)\; x_{i-1}, x_{i} \mbox{ are regular nilpotent.} \\
&(2)\; \fz_{\fg_{i-1}}(x_{i-1})\cap\fz_{\fg_{i}}(x_{i})=0.\end{split}
\end{equation}
 We proceed by finding the Borels in $\fg_{i}$ which contain elements satisfying (1) and (2), and we then use these Borels to construct the Borels of $\fg$ which contain elements of $\Phi^{-1}(0)_{sreg}$.  
 
 Let $K_{i}:=GL(i-1,\C)\times GL(1,\C)\subset GL(i,\C)$ be the group of invertible block diagonal matrices with an $(i-1)\times (i-1)$ block in the upper left corner and a $1\times 1$ block in the lower right corner.  Let $\B_{i}$ be the flag variety of $\fg_{i}$.  Then $K_{i}$ acts on $\B_{i}$ by conjugation with finitely many orbits (see for example \cite{Sp}).  We observe that the conditions (1) and (2) in (\ref{eq:theconditions}) are $\Ad(K_{i})$-equivariant.  Thus, the problem of finding the Borel subalgebras of $\fg_{i}$ containing elements satisfying these conditions reduces to the problem of studying the conditions for
a representative in each $K_{i}$-orbit.  In this section, we find all $K_{i}$-orbits $Q_{i}$ through Borel subalgebras containing such elements, and in the process reveal some new facts about the geometry of $K_{i}$-orbits on $\B_{i}$.  In the following sections, we explain how to link the orbits $Q_{i}$ together for $i=2,\dots, n$ to produce the Borel subalgebras of $\fg$ that contain elements of $\Phi^{-1}(0)_{sreg}$ and use these Borels to study the geometry of the fiber $\Phi^{-1}(0)_{sreg}$.

For concreteness, let us fix $i=n$, so that $K_{n}=GL(n-1,\C)\times GL(1,\C)$ and $\B_{n}$ is the flag variety of $\fgl(n,\C)$.  For $\fb\in \mathcal{B}_{n}$, let $K_{n}\cdot \fb$ denote the $K_{n}$-orbit through $\fb$.  We analyze each of the conditions in (\ref{eq:theconditions}) in turn.


\begin{thm}\label{thm:1closed}(\cite{CEKorbs}, Proposition 3.6)
Suppose $x\in\fg$ satisfies condition (1) in (\ref{eq:theconditions}) and that $x\in\fb$, with $\fb\subset\fg$ a Borel subalgebra of $\fg$.  Then $\fb\in Q$, where $Q$ is a closed $K_{n}$-orbit.
\end{thm}

Theorem \ref{thm:1closed} follows from a stronger result.  The group $K_{n}$ is the group of fixed points of the involution $\theta$ on $G$, where $\theta(g)=cgc^{-1}$ with $c=diag[1,\dots, 1, -1]$.  Let $\fk_{n}=Lie(K_{n})$, so that $\fk_{n}$ is the Lie algebra of block diagonal matrices $\fk_{n}=\fgl(n-1,\C)\oplus\fgl(1,\C)$.  Then $\fg=\fk_{n}\oplus\fp_{n}$, where $\fp_{n}$ is the $-1$-eigenspace for the involution $\theta$ on $\fg$.  Let $\pi_{\fk_{n}}:\fg\to\fk_{n}$ be the projection of $\fg$ onto $\fk_{n}$ along $\fp_{n}$, and let $\mathcal{N}_{\fk_{n}}$ be the nilpotent cone in $\fk_{n}$.

\begin{thm}\label{thm:strong}(\cite{CEKorbs}, Theorem 3.7)
Let $\fb\subset\fg$ be a Borel subalgebra and let $\fn=[\fb,\fb]$, with $\fn^{reg}$ the regular nilpotent elements in $\fb$.  Suppose that $\fb\in Q$ with $Q$ a $K_{n}$-orbit in $\B_{n}$ which is not closed.  Then $\pi_{\fk_{n}}(\fn^{reg})\cap \mathcal{N}_{\fk_{n}}=\emptyset.$ 
\end{thm}

\begin{rem}\label{rem:bigproof}
By the $K_{n}$-equivariance of the projection $\pi_{\fk_{n}}:\fg\to\fk_{n}$, it suffices to prove Theorem \ref{thm:strong} for a representative of the $K_{n}$-orbit $Q$.  Standard representatives are given by the Borel subalgebras
$\fb_{i,j}$
discussed later in Notation \ref{nota:qijdef} and Example \ref{ex:qijthetaroots}.
  Let $\fb=\fb_{i,j}$ be such a representative.  To compute $\pi_{\fk_{n}}(\fn^{reg})$, one needs to understand the action of $\theta$ on the roots of $\fb$ with respect to a $\theta$-stable Cartan $\fh^{\prime}\subset\fb$.  In general, this action is difficult to compute.  It is easier to replace the pair $(\fb, \theta)$ with an equivalent pair $(\fb_{+},\theta^{\prime})$ where $\fb_{+}\subset\fg$ is the standard Borel subalgbera of upper triangular matrices and $\theta^{\prime}$ is an involution of $\fg$ which stabilizes the standard Cartan subalgebra of diagonal matrices $\fh\subset\fb_{+}$.  We then prove the statement of the theorem for the pair $(\fb_{+},\theta^{\prime})$.  The construction and computation of the involution $\theta^{\prime}$ is explained in detail 
in Equation (\ref{eq:newinv}) and Example \ref{ex:qijthetaroots}, where it is denoted by $\theta_{\hv}$ and $\theta_{\widehat{v_{i,j}}}$ respectively.
\end{rem}

Theorem \ref{thm:1closed} permits us to focus only on closed $K_{n}$-orbits.  There are $n$ such orbits in $\B_{n}$, two of which are $Q_{+,n}=K_{n}\cdot\fb_{+}$, the orbit of the $n\times n$ upper triangular matrices, and $Q_{-,n}=K_{n}\cdot\fb_{-}$, the orbit of the $n\times n$ lower triangular matrices (see Example \ref{ex:closed}).  We now study the second condition in (\ref{eq:theconditions}). 
\begin{prop}\label{p:second}
Let $Q=K_{n}\cdot\fb$ be a closed $K_{n}$-orbit and let $x\in\fn=[\fb,\fb]$ satisfy condition (2) in (\ref{eq:theconditions}).  Then $Q=Q_{+,n}$ or $Q=Q_{-,n}$.  
\end{prop}
This is an immediate consequence of the following result.  Recall the projection $\pi_{n-1}:\fg\to\fg_{n-1}$ defined by $\pi_{n-1}(x)=x_{n-1}.$ 

\begin{prop}\label{p:centralizers} (\cite{CEKorbs}, Proposition 3.8)
Let $\fb\subset\fg$ be a Borel subalgebra
 that generates a closed $K_{n}$-orbit $Q$, 
which is neither the orbit of the upper nor the lower triangular matrices.
  Let $\fn=[\fb,\fb]$ and let $\fn_{n-1}:=\pi_{n-1}(\fn)$.
  Let $\fz_{\fg}(\fn)$ denote the centralizer of $\fn$ in $\fg$ and 
let $\fz_{\fg_{n-1}}(\fn_{n-1})$ denote the centralizer of $\fn_{n-1}$ in $\fg_{n-1}$.  Then 
\begin{equation}\label{eq:nilcentralizer}
\fz_{\fg_{n-1}}(\fn_{n-1})\cap \fz_{\fg}(\fn)\neq 0.
\end{equation}  
\end{prop}

\begin{rem}
We note that the projection $\pi_{n-1}:\fg\to\fg_{n-1}$ is $K_{n}$-equivariant, so that it suffices to prove Equation (\ref{eq:nilcentralizer}) for a representative $\fb$ of the closed $K_{n}$-orbit $Q$.  We can take $\fb$ to be one of the representatives given below in Example \ref{ex:closed}.
\end{rem}
 For any $i=2,\dots, n$, let $Q_{+,i}$ denote the $K_{i}$-orbit of the $i\times i$ upper triangular matrices in $\B_{i}$ and let $Q_{-,i}$ denote the $K_{i}$-orbit of the $i\times i$ lower triangular matrices in $\B_{i}$.  Combining the results of Theorem \ref{thm:1closed} and Proposition \ref{p:second}, we obtain:
\begin{thm}\label{thm:both}
Let $x\in \fg_{i}$ satisfy the two conditions in (\ref{eq:theconditions}) and suppose that $x\in\fb$, with $\fb\subset\fg_{i}$ a Borel subalgebra.  Then $K_{i}\cdot \fb=Q_{+,i}$ or $K_{i}\cdot \fb=Q_{-,i}$.
\end{thm}

\subsection{ Constructing Borel subalgebras out of $K_{i}$-orbits}
\label{s:constructborels}
  In this section, we explain how to link together the $K_{i}$-orbits $Q_{+,i}$ and $Q_{-,i}$ for $i=2,\dots, n$ to construct all the Borel subalgebras containing elements of $\Phi^{-1}(0)_{sreg}$.  The key to the construction is the following lemma.

\begin{lem}\label{l:closed}(\cite{CEKorbs}, Proposition 4.1)
Let $Q$ be a closed $K_{n}$-orbit in $\B_{n}$ and let $\fb\in Q$.  Then $\pi_{n-1}(\fb)\subset\fg_{n-1}$ is a Borel subalgebra.  
\end{lem}


We can use Lemma \ref{l:closed} to give an inductive construction of
 special subvarieties of $\B_{n}$ by linking together closed $K_{i}$-orbits $Q_{i}$ for $i=2,\dots, n$.  For this construction, we view $K_{i}\subset K_{i+1}$ by embedding $K_{i}$ in the upper left corner of $K_{i+1}$.  We also make use of the following notation.  If $\fm\subset\fg$ is a subalgebra, we denote by $\fm_{i}$ the image of $\fm$ under the projection $\pi_{i}:\fg\to\fg_{i}$.  

Suppose we are given a sequence $\mathcal{Q}=(Q_{2},\dots, Q_{n})$ with $Q_{i}$ a closed $K_{i}$-orbit in $\mathcal{B}_{i}$.  We call $\calQ$ a sequence of closed $K_{i}$-orbits.  For $\fb\in Q_{n}$, $\fb_{n-1}$ is a Borel subalgebra by
 Lemma \ref{l:closed}.  Since $K_{n}$ acts transitively
on $\mathcal{B}_{n-1}$, there is $k\in K_{n}$ 
such that $\Ad(k)\fb_{n-1}\in Q_{n-1}$ and the variety 
$$
X_{Q_{n-1}, Q_{n}}:=\{\fb\in \B_{n}:\fb\in Q_{n},\,  \fb_{n-1}\in Q_{n-1}\} 
$$
is nonempty.  Lemma \ref{l:closed} again implies that $(\Ad(k)\fb_{n-1})_{n-2}=(\Ad(k)\fb)_{n-2}$ is a Borel subalgebra in $\fg_{n-2}$, so that there exists an $l\in K_{n-1}$ such that $\Ad(l)(\Ad(k)\fb)_{n-2}\in Q_{n-2}$.  Since $K_{n-1}\subset K_{n}$, the variety 
$$X_{Q_{n-2}, Q_{n-1}, Q_{n}}=\{\fb\in\B_{n}: \fb\in Q_{n},\,  \fb_{n-1}\in Q_{n-1}, \, \fb_{n-2}\in Q_{n-2}\} $$
 is nonempty.  Proceeding in this fashion, we can define a nonempty closed subvariety of $\B_{n}$ by 
 \begin{equation}\label{eq:orbittower}
 \orbittower=\{\fb\in \B_{n}: \fb_{i}\in Q_{i},\, 2\leq i\leq n\}.
 \end{equation}


 \begin{thm}\label{thm:areborels}(\cite{CEKorbs}, Theorem 4.2)
 Let $\calQ=(Q_{2},\dots, Q_{n})$ be a sequence of closed $K_{i}$-orbits.  Then the variety $\orbittower$ is a single Borel subalgebra of $\fg$ that contains the standard Cartan subalgebra of diagonal matrices.  Moreover, if $\fb\subset\fg$ is a Borel subalgebra which contains the diagonal matrices, then $\fb=\orbittower$ for some sequence of closed $K_{i}$-orbits $\calQ$.  
 \end{thm}
 
 \begin{nota}
In light of Theorem \ref{thm:areborels}, we refer to the Borel subalgebras $\orbittower$ as $\borel$ for the remainder of the discussion.
\end{nota}
 
\subsection{Borels containing elements of $\Phi^{-1}(0)_{sreg}$}

Now we can at last describe the Borel subalgebras of $\fg$ that contain elements of $\Phi^{-1}(0)_{sreg}$ and use these to determine the irreducible component decomposition of $\Phi^{-1}(0)_{sreg}$ explicitly.  Since $x\in\Phi^{-1}(0)_{sreg}$ if and only if $x_{i}\in\fg_{i}$ satisfies the two conditions in (\ref{eq:theconditions}) for all $i=2,\dots, n$, Theorem \ref{thm:both} implies:
\begin{prop}\label{p:theborels}(\cite{CEKorbs}, Theorem 4.5)
Let $x\in\Phi^{-1}(0)_{sreg}$.  Then $x\in\borel$, where the sequence of closed $K_{i}$-orbits $\mathcal{Q}=(Q_{2}, \dots, Q_{n})$ has $Q_{i}=Q_{+,i}$ or $Q_{i}=Q_{-,i}$ for each $i=2,\dots, n$.  
\end{prop}



 \begin{exam}\label{ex:first}
It is easy to describe explicitly these Borel subalgebras. For example, for $\fg=\fgl(3,\C)$ there are four such Borel subalgebras:
  $$
\begin{array}{cc}
\begin{array}{c}
\mathfrak{b}_{Q_{-},Q_{-}}=\left[\begin{array}{ccc} 
h_{1} & 0 & 0 \\
a_{1}& h_{2} &0\\
a_{2}& a_{3} & h_{3}\end{array}\right]
 \end{array}
&
\begin{array}{c}
\mathfrak{b}_{Q_{+}, Q_{+}}=\left[\begin{array}{ccc} 
h_{1} & a_{1} & a_{2} \\
0 & h_{2} & a_{3}\\
0 & 0 & h_{3}\end{array}\right]
\end{array}
\\
& \\
\begin{array}{c}
\mathfrak{b}_{Q_{+},Q_{-}}=\left[\begin{array}{ccc} 
h_{1}& a_{1} & 0 \\
0 & h_{2} & 0\\
a_{2} & a_{3} & h_{3}\end{array}\right]
\end{array}
&
\begin{array}{c}
\mathfrak{b}_{Q_{-},Q_{+}}=\left[\begin{array}{ccc} 
h_{1} & 0 & a_{1} \\
a_{2} & h_{2}& a_{3}\\
0 & 0 & h_{3}\end{array}\right]
\end{array}
\end{array}_{\mbox{,}}
$$
where $a_{i}, h_{i}\in\C$.  
\end{exam}

We can use these Borel subalgebras to describe the fiber $\Phi^{-1}(0)_{sreg}$.  
Let $\nilrad^{reg}$ be the subvariety of regular nilpotent elements of $\borel$.  
Proposition \ref{p:theborels} implies,
\begin{equation}\label{eq:inclusion}
\Phi^{-1}(0)_{sreg}\subseteq\coprod_{\calQ} \nilrad^{reg},
\end{equation}
where $\calQ=(Q_{2},\dots, Q_{n})$ ranges over all $2^{n-1}$ sequences
 where $Q_{i}=Q_{+,i}$ or $Q_{-,i}$.  We note that the union on the right side of (\ref{eq:inclusion}) is disjoint, since a regular nilpotent element is contained in a unique Borel subalgebra (see for example \cite{CG}, Proposition 3.2.14).  We claim that the inclusion in (\ref{eq:inclusion}) is an equality and that the right side of (\ref{eq:inclusion}) is an irreducible component decomposition of the variety $\Phi^{-1}(0)_{sreg}$.  The key observation is the converse to Proposition \ref{p:theborels}.
 \begin{prop}(\cite{CEKorbs}, Proposition 3.11, Theorem 4.5)
 Let $\mathcal{Q}=(Q_{2},\dots, Q_{n})$ be a sequence of closed $K_{i}$-orbits with $Q_{i}=Q_{+,i}$ or $Q_{-,i}$.  Let $\nilrad^{reg}$ be the regular nilpotent elements of $\borel$.  Then $\nilrad^{reg}\subset \Phi^{-1}(0)_{sreg}$.
 \end{prop}



 Thus, the variety $\nilrad^{reg}$ is an irreducible subvariety of $\Phi^{-1}(0)_{sreg}$ of dimension $\dim \nilrad=\dn$.  It follows from Theorem \ref{thm:irred} that $\nilrad^{reg}$ is an open subvariety of a unique irreducible component, $\mathcal{Y}$ of $\Phi^{-1}(0)_{sreg}$.  But then by (\ref{eq:inclusion}), we have
\begin{equation*} 
\mathcal{Y}=\coprod_{\calQ^{\prime}}\fn_{\calQ^{\prime}}^{reg},
\end{equation*}
where the disjoint union is taken over a subset of the set of all sequences $(Q_{2}^{\prime},\dots, Q_{n}^{\prime})$ with $Q^{\prime}_{i}=Q_{+,i}$ or $Q_{-,i}$.  Since $\mathcal{Y}$ is irreducible, we must have $\nilrad^{reg}=\mathcal{Y}$.  This yields the main theorem of \cite{CEKorbs}.

\begin{thm}\label{thm:nilpotent}(\cite{CEKorbs}, Theorem 4.5)
 The irreducible component decomposition of the variety $\Phi^{-1}(0)_{sreg}$ is
\begin{equation}\label{eq:nilfibre}
\Phi^{-1}(0)_{sreg}=\coprod_{\calQ} \nilrad^{reg},
\end{equation}
where $\calQ=(Q_{2},\dots, Q_{n})$ ranges over all $2^{n-1}$ sequences
 where $Q_{i}=Q_{+,i}$ or $Q_{-,i}$.
The $A$-orbits in $\Phi^{-1}(0)_{sreg}$ are exactly the irreducible
components
$\nilrad^{reg}$, for $\calQ$ as above.
\end{thm}



The description of $\Phi^{-1}(0)_{sreg}$ in Equation (\ref{eq:nilfibre}) is much more explicit than the one given in Corollary \ref{c:nilp}, where the components are described as orbits of the group $Z=(\C^{\times})^{n-1}\times\C^{\dn-n+1}$ where $Z$ acts via the formula in Equation (\ref{eq:newact}).  In fact, we can describe easily the varieties $\nilrad^{reg}\cong (\C^{\times})^{n-1}\times\C^{\dn-n+1}$. 

\begin{exam}
For $\fg=\fgl(3,\C)$, Theorem \ref{thm:nilpotent} implies that the four $A$-orbits in $\Phi^{-1}(0)_{sreg}$ are the regular nilpotent elements of the four Borel subalgebras given in Example \ref{ex:first}.  
  $$
\begin{array}{cc}
\begin{array}{c}
\mathfrak{n}_{Q_{-},Q_{-}}^{reg}=\left[\begin{array}{ccc} 
0 & 0 & 0 \\
a_{1}& 0 &0\\
a_{3}& a_{2} & 0\end{array}\right]
 \end{array}
&
\begin{array}{c}
\mathfrak{n}_{Q_{+}, Q_{+}}^{reg}=\left[\begin{array}{ccc} 
0 & a_{1} & a_{3} \\
0 & 0 & a_{2}\\
0 & 0 & 0\end{array}\right]
\end{array}
\\
& \\
\begin{array}{c}
\mathfrak{n}_{Q_{+},Q_{-}}^{reg}=\left[\begin{array}{ccc} 
0& a_{1} & 0 \\
0 & 0& 0\\
a_{2} & a_{3} & 0\end{array}\right]
\end{array}
&
\begin{array}{c}
\mathfrak{n}_{Q_{-},Q_{+}}^{reg}=\left[\begin{array}{ccc} 
0 & 0 & a_{1} \\
a_{2} & 0& a_{3}\\
0 & 0 & 0\end{array}\right]
\end{array}
\end{array}_{\mbox{,}}
$$
where $a_{1}, \, a_{2}\in\C^{\times}$ and $a_{3}\in\C$.  
\end{exam}

\begin{rem}\label{r:reptheory}
 We note that the $2^{n-1}$ Borel subalgebras appearing in Theorem \ref{thm:nilpotent} 
are exactly the Borel subalgebras $\fb$ with the property that 
each projection of $\fb$ to
$\fgl(i,\C)$ for $i=2, \dots, n$ is a Borel subalgebra of $\fg_{i}$
whose $K_{i}$-orbit in $\B_{i}$ is related via the Beilinson-Bernstein
correspondence to Harish-Chandra modules for the pair $(\fgl(i,\C),
K_{i})$ coming from holomorphic and anti-holomorphic discrete series.
It would be interesting to relate our results to representation
theory, especially to work of Kobayashi \cite{Kobsur}.
For more on the relation between geometry of orbits of a symmetric subgroup
and Harish-Chandra modules, see
\cite{Vg}, \cite{HMSW}, \cite{Collingwood}.
\end{rem}

\subsection{Strongly Regular Elements and Borel subalgebras}

It would be interesting to study strongly regular fibers $\Phi^{-1}(c)_{sreg}$ for arbitrary $c\in\C^{\dnone}$ using the geometry of $K_{i}$-orbits on $\B_{i}$.  The following result is a step in this direction.
\begin{thm}\label{thm:sregandborels}(\cite{CEKorbs}, Theorem 5.3) 
Every Borel subalgebra $\fb\subset\fg$ contains strongly regular elements. 
\end{thm}
We briefly outline the proof of Theorem \ref{thm:sregandborels}.  For complete details see \cite{CEKorbs}, Section 5.  For ease of notation, we denote the flag variety $\B_{n}$ of $\fgl(n,\C)$ by $\B$.  Let $\fh\subset\fg$ denote the standard Cartan subalgebra of diagonal matrices and let $H$ be the corresponding Cartan subgroup.  Define 
$$
\B_{sreg}=\{\fb\in \B:\; \fb\cap \fg_{sreg}\neq \emptyset\}.  
$$
We want to show that $\B_{sreg}=\B$.  Consider the variety $Y=\B\setminus \B_{sreg}$.  We show that $Y$ is closed and $H$-invariant.  Let $\fb\in Y$ and consider its $H$-orbit, $H\cdot \fb$.  Since $Y$ is closed $\overline{H\cdot\fb}\subset Y$.  We know that $\overline{H\cdot \fb}$ contains a closed $H$-orbit.  But the closed $H$-orbits on $\B$ are precisely the Borels subalgebras $\fb$ which contain the Cartan subalgebra $\fh$ (\cite{CG}, Lemma 3.1.10).  Thus, it suffices to show that no Borel subalgebra $\fb$ with $\fh\subset \fb$ can be contained in $Y$.  This can be shown using the characterization of such Borels as $\borel$, with $\mathcal{Q}=(Q_{2},\dots, Q_{n})$ a sequence of closed $K_{i}$-orbits (see Theorem \ref{thm:areborels}) and properties of closed $K_{i}$-orbits (see \cite{CEKorbs}, Proposition 5.2).

\section{The geometry of $K$-orbits on the flag variety}\label{s:Kgeom}

Proofs of the results discussed in Section \ref{s:Korbs} require an
understanding of aspects of the geometry and parametrization of
$K_n$-orbits on the flag variety $\mathcal{B}_{n}$ of $\fgl(n,\C)$.
In this section, we develop the general theory of orbits of a symmetric subgroup $K$ of an algebraic group $G$ acting on the flag variety $\B$ of $G$.  We obtain representatives for the $K$-orbits on $\B$ and compute the involution $\theta^{\prime}$ mentioned in Remark \ref{rem:bigproof} for any $K$-orbit.  Along the way, we apply the general theory to the specific example of $G=GL(n,\C)$ and $K=GL(n-1,\C)\times GL(1,\C)$, providing the details behind the computations of \cite{CEKorbs}, Section 3.1.  
See the papers \cite{RS},\cite{RSexp}, and \cite{Vg} for results concerning
 orbits of a general symmetric subgroup on the flag variety.




\subsection{Parameterization of $K$-orbits on $G/B$}\label{s:para}
Let $G$ be reductive group over $\C$ such that $[G,G]$ is simply connected.  Let $\theta:G\to G$ be a holomorphic involution, and we also refer to the
differential of $\theta$ as $\theta:\fg \to \fg$.
Since $\theta:\fg\to\fg$ is a Lie algebra homomorphism, it preserves $[\fg, \fg]$ and the Killing form $<\cdot, \cdot>$ of $\fg$.  Let $K=G^{\theta}$ and assume that the fixed set $(Z(G)^0)^{\theta}$ is connected, where
$Z(G)^0$ is the identity connected component of the center of $G$. 
 Then by a theorem of Steinberg (\cite{St}, Corollary 9.7), $K$ is connected.

Let $\B$ be the flag variety of $\fg$, and recall that if $B$ is a Borel 
subgroup of $G$, the morphism $G/B \to \B$, $gB\mapsto \Ad(g)\fb$, where $\fb=Lie(B)$, is
a $G$-equivariant isomorphism $G/B\cong \B$.
The involution $\theta$ acts on the variety $\T$ of Cartan subalgebras
of $\fg$ by $\ft \mapsto \theta(\ft)$ for $\ft\in\mathcal{T}$, and the fixed set $\T^{\theta}$
is the variety of $\theta$-stable Cartan subalgebras.  We consider the variety 
$$
\mathcal{C}=\{(\fb,\ft)\in\B\times\T: \ft\subset \fb\}.
$$
Then $G$ acts on $\mathcal{C}$ through the adjoint action,
 and the subvariety $\mathcal{C}_{\theta}=\mathcal{C}\cap (\B\times \T^{\theta})$ is $K$-stable.  Consider the $G$-equivariant map $\pi:\mathcal{C}\to\B$ given by projection onto the first coordinate, $\pi(\fb, \ft)=\fb$.  It induces
a map 
\begin{equation}\label{eq:gamma}
\gamma: K\backslash \mathcal{C}_{\theta}\to K\backslash\B, \ \gamma(K\cdot(\fb, \ft)) = K\cdot \fb
\end{equation}
from the set of $K$-orbits on $\mathcal{C}_{\theta}$ to the set of $K$-orbits on $\B$. 

\begin{prop}\label{prop:gamma}
The map $\gamma$ is a bijection.  
\end{prop}

For a proof of this proposition, we refer the reader to \cite{RSexp}, Proposition 1.2.1.   We summarize the main ideas.  To show the map $\gamma$ is surjective, it suffices to show that every Borel subalgebra contains a $\theta$-stable Cartan.  This follows from \cite{St}, Theorem 7.5.  To show that the map is injective, it suffices to show that if $\ft, \ft^{\prime}$ are $\theta$-stable Cartan subalgebras of a Borel subalgebra $\fb$, then $\ft$ and $\ft^{\prime}$ are $K\cap B$-conjugate, which is verified in \cite{RSexp}.

Throughout the discussion, we will fix a $\theta$-stable Borel $\fb_{0}$ and $\theta$-stable Cartan $\ft_{0}\subset\fb_{0}$.  Such a pair exists by \cite{St}, Theorem 7.5, and is called a {\it standard pair}.  Let $N=N_{G}(T_{0})$ be the normalizer of $T_{0}$, where $T_{0}$ is the Cartan subgroup with Lie algebra $\ft_{0}$.   We consider the map $\zeta_{0}: G\to \mathcal{C}$ given by $\zeta_{0}(g)=(\Ad(g) \fb_{0}, \Ad(g)\ft_{0})$, which is clearly $G$-equivariant with
respect to the left translation action on $G$ and the adjoint action on
$\mathcal{C}$.  It is easy to see that $\zeta_{0}$ is constant on left $T_{0}$-cosets, and
induces an isomorphism of varieties  
\begin{equation}\label{eq:zeta}
\zeta: G/T_{0}\to \mathcal{C}.   
\end{equation}
To parameterize the $K$-orbits on $\B$ using Proposition \ref{prop:gamma}, we introduce the variety $\V=\zeta_{0}^{-1}(\mathcal{C}_{\theta})$.  It is easy to show that $\V$ is the set  
\begin{equation}\label{eq:varietyV}
\V=\{g\in G:\, g^{-1}\theta(g)\in N\}.
\end{equation}
 By Equation (\ref{eq:zeta}) and the $G$-equivariance of the map $\zeta_{0}$, 
it follows that the morphism $\zeta$ induces a bijection,  
 \begin{equation}\label{eq:doublecosets}
 \zeta:K\backslash\V/ T_{0}\to K\backslash \mathcal{C}_{\theta},
 \end{equation}
 which we also denote by $\zeta$. Combining Equation (\ref{eq:doublecosets}) 
with Proposition \ref{prop:gamma}, we obtain the following useful 
parametrization of $K$-orbits on $\B$ (cf. \cite{RSexp}, Proposition 1.2.2).  
 \begin{prop}\label{prop:bijections}
 There are natural bijections 
 $$
 K\backslash \V /T_{0}\leftrightarrow K\backslash \mathcal{C}_{\theta}\leftrightarrow K\backslash \B \leftrightarrow K\backslash G/ B_{0}.
 $$
 \end{prop}
 
 
  Let $V$ denote the set of $(K, T_{0})$-double cosets in $\V$.  By \cite{Sp}, Corollary
 4.3, $V$ is a finite set and hence:
  $$
  \mbox{\textbf {The number of $K$-orbits on $\B$ is finite. }}
  $$
  
  \begin{nota}
  For $v\in V$, let $\hv \in \V$ denote a representative, so that
 $v=K\hv T$.  Denote the corresponding  $K$-orbit in $\B$ by $K\cdot \fb_{\hv}$, where $\fb_{\hv}=\Ad(\hv)\cdot\fb_{0}$.
  \end{nota}
  
  We end this section with a  discussion of how $\theta$ acts on the root decomposition of $\fg$ with respect to a $\theta$-stable Cartan subalgebra $\ft$.  

  \begin{dfn}\label{defn:roots}
For $(\fb, \ft) \in \mathcal{C}_{\theta}$ and $\alpha \in \Phi = \Phi (\fg, \ft)$,
let $e_\alpha \in \fg_{\alpha}$ be a root vector in the corresponding root space.
We say that $\alpha$ is positive for $(\fb, \ft)$ if $\fg_{\alpha}
\subset \fb$.
We define the {\it type} of $\alpha$ for the pair $(\fb, \ft)$ with respect to $\theta$ as follows.  
  \begin{enumerate}
  \item If $\theta(\alpha)=-\alpha$, then $\alpha$ is said to be \emph{real}.
  \item If $\theta(\alpha)=\alpha$, then $\alpha$ is said to be \emph{imaginary}.  In this case, there are two subcases:
  \begin{enumerate}
  \item If $\theta(e_{\alpha})=e_{\alpha}$, then $\alpha$ is said to be \emph{compact imaginary}.
  \item If $\theta(e_{\alpha})=-e_{\alpha}$, then $\alpha$ is said to be \emph{non-compact imaginary}.
  \end{enumerate}
  \item If $\theta(\alpha)\neq \pm \alpha$, then $\alpha$ is said to \emph{complex}.  If 
also $\alpha$ and $\theta(\alpha)$ are both positive, we say $\alpha$ is  complex $\theta$-stable.
  \end{enumerate}

  \end{dfn}
  
\begin{rem}\label{rem:thetastableroots}
Let $\alpha$ be a positive root.  Then
$\theta(\alpha)$ is positive if and only if $\alpha$ is imaginary
or complex $\theta$-stable.
\end{rem}

For $v\in V$ with representative $\hv \in \V$, we define a new involution
by the formula,
\begin{equation}\label{eq:newinv}
\theta_{\hv}=\Ad(\hv^{-1})\circ\theta\circ \Ad(\hv)=\Ad(\hv^{-1}\theta(\hv))\circ\theta.
\end{equation}
Note that $\theta_{\hv}(\ft_0)=\ft_0$, and consider the induced action
of $\theta_{\hv}$ on $\Phi(\fg, \ft_0)$.

\begin{dfn}\label{defn:rootsforv}  Let $\alpha \in \Phi(\fg, \ft_0),\, v\in V$, and $\hv\in\V$ be a representative for $v$.  
We define \emph{the type of the root} $\alpha$ \emph{for} $v$ to be the type of the root $\alpha$ for the pair $(\fb_{0},\ft_{0})$ with 
respect to the involution $\theta_{\hv}$.

\end{dfn}

For example, a root $\alpha$ is imaginary for $v$ if and only if $\theta_{\hv}(\alpha)=\alpha$.   Note that if $k\hv t$ is a different representative for $v$, then
$\theta_{k \hv t} = \Ad(t^{-1}) \circ \theta_{\hv} \circ \Ad(t)$.  It follows
easily that the type of $\alpha$ for $v$ does not depend on the choice of a representative
$\hv$.  Further, the involution $\theta_{\hv}$ of $\Phi(\fg, \ft_0)$ does not depend
on the choice of $\hv$, and we refer to $\theta_{\hv}$ as the involution of 
associated to the orbit $v$.

 For $v\in V$ and $\fb_{\hv}=\Ad(\hv)\cdot \fb_{0}$, consider the $\theta$-stable 
Cartan subalgebra $\ft^{\prime} = \Ad(\hv)\cdot \ft_{0} \subset \fb_{\hv}$.
For $\alpha\in\Phi(\fg,\ft_{0})$, we define  
$\Ad(\hv)\alpha := \alpha\circ\Ad(\hv^{-1})\in\Phi(\fg,\ft^{\prime}).$  

\begin{prop}\label{prop:rootsforv}
For $\alpha \in \Phi(\fg, \ft_0)$, the type of $\alpha$ for $v$ is the
same as the type of $\Ad(\hv)\alpha$ for the pair $(\fb_{\hv}, \ft^{\prime})$ with respect
to $\theta$.
\end{prop}

\begin{proof} 
This follows easily from the identity $\theta \circ \Ad(\hv)
= \Ad(\hv)\circ\theta_{\hv}.$
\end{proof}

By Proposition \ref{prop:rootsforv}, we may compute the action of
$\theta$ on the positive roots in $\Phi(\fg, \ft^{\prime})$ for the pair $(\fb_{\hv}, \ft^{\prime})$ using
the involution $\theta_{\hv}$ on our standard positive system $\Phi^{+}(\fg,\ft_{0})$
in $\Phi(\fg, \ft_0)$.






\begin{rem}
We also denote the corresponding involution on $G$ by $\theta_{\hv}$.  By abuse of notation, we denote conjugation on $G$ by $\Ad$, i.e., for $g,\, h\in G;\; \Ad(g)h=ghg^{-1}$.  Thus $\theta_{\hv}:G\to G$ is also given by the formula in Equation (\ref{eq:newinv}).  Its differential at the identity is $\theta_{\hv}:\fg\to\fg$.  
\end{rem}

  \subsection{The $W$-action on $V$}\label{s:Weyl}
 The fact that $K$-orbits on the flag variety have representatives coming
from $\V$ was used by Springer \cite{Sp} to associate a Weyl group element
$\phi(v)$ to the $K$-orbit indexed by $v\in V$. The element $\phi(v)$
plays a crucial role in understanding the action of the involution $\theta_{\hv}$ associated to $v$ on the roots for the standard pair $\Phi(\fg,\ft_{0})$.  


We first consider the map $\tau:G \to G$ given
by $\tau(g)=g^{-1}\theta(g).$ Note that $\tau^{-1}(N)=\V$. Then following \cite{Sp}, Section 4.5,
we define for $v=K\hv T_{0}$
\begin{equation}\label{eq:mapphi}
\phi(v)=\tau(\hv)T_0 \in N/T_0 = W.
\end{equation}
We refer to the map $\phi$ as the Springer map and $\phi(v)$ as the Springer invariant of $v\in V$.  It is easy to check that $\phi(v)$ is independent of the choice of representative $\hv$.

  


The Springer map is not injective, but we can study its fibers using an action of $W$ on $V$, which we now describe.  The group $N$ acts on $\V$ on the left by $n\cdot \hat{v}=\hat{v} n^{-1}$ for $\hat{v}\in\V$ and $n\in N$.  This action induces a $W$-action on $V$ given by
 \begin{equation}\label{eq:crossaction}
 w\times v:= K\hv \dot{w}^{-1}T_{0},
 \end{equation}
 where $\hv\in\V$ is a representative of $v\in V$ and $\dot{w}\in N$ is a representative of $w\in W$.  It is easy to check that the formula in Equation (\ref{eq:crossaction}) does not depend on the choice of representatives $\dot{w}$ or $\hv$.  We refer to this action as the \emph{cross action} of $W$ on $V$.  The Springer map intertwines the cross action of $W$ on $V$ with a certain twisted action of $W$ on itself.  We note that since $T_{0}$ is $\theta$-stable, $\theta$ acts on $N$ and hence on $W$.  We define the twisted conjugation action 
of $W$ on itself by:
  \begin{equation}\label{eq:twistedact}
  w^{\prime}*w=w^{\prime} w\theta((w^{\prime})^{-1}), \mbox{ for }  w, w^{\prime} \in W.
  \end{equation}

 
  
\begin{prop}\label{prop:actions}
\begin{enumerate}
\item The Springer map $\phi:V\to W$ is $W$-equivariant with respect the cross action on $V$ and the twisted $W$-action on $W$. 
\item (\cite{RS}, Proposition 2.5) Suppose for $v,\, v^{\prime} \in V$, we have $\phi(v)=\phi(v^{\prime})$.  Then $v^{\prime}=w\times v$ for some $w\in W$.  
\end{enumerate}
\end{prop}
Part (1) is an easy calculation using the definition of $\phi$.  Part (2) is non-trivial and relies on many of the results of \cite{RS}, Section 2.



\subsection{Closed $K$-orbits on $\B$}\label{s:closed}

In this section, we use the properties of the Springer map developed in the previous section to find representatives for the closed $K$-orbits on $\B$ and describe the involution $\theta_{\hv}$ associated to such orbits.  

Since $\theta$ acts on $W$, we can consider the $W$-fixed point subgroup,
 $W^{\theta}$.  By \cite{Rich}, Lemma 5.1, $T_0 \cap K$ is a maximal torus
of $K$, and by \cite{Rich}, Lemma 5.3, the subgroup $N_K(T_{0} \cap K)
\subset N_G(T_{0})$.  It follows that the group homomorphism
$N_K(T_{0} \cap K)/(T_0 \cap K) \to N_G(T_{0})/T_{0}$ is
injective. Hence, we may regard $W_K$ as a subgroup of $W$, and it is
easy to see that it has image in $W^{\theta}$.

\begin{thm}\label{thm:closed}
There is a one-to-one correspondence between the set of closed $K$-orbits on $\B$ and the coset space $W^{\theta}/ W_{K}$.  The correspondence is given by: 
\begin{equation}\label{eq:correspond}
w\, W_{K}\to K \dot{w}^{-1} T_{0},
\end{equation}
for $\dot{w}\in N$ a representative of $w\in W^{\theta}$. 
\end{thm}


To prove Theorem \ref{thm:closed}, we describe equivalent conditions for a $K$-orbit on $\B$ to be closed.  We begin with the following lemma (see \cite{BH}, Lemma 3).

\begin{lem}\label{l:2closed}
Let $B\subset G$ be a Borel subgroup.  Then the following statements are equivalent.
\begin{enumerate}
\item The Borel subgroup $B$ is $\theta$-stable.
\item The subgroup $(B\cap K)^{0}$ is a Borel subgroup of $K$, where $(B\cap K)^{0}$ denotes the identity component of $B\cap K$.
\end{enumerate}
\end{lem}

Let $v_{0}\in V$ correspond to the $K$-orbit $K\cdot\fb_{0}$ so that $v_{0}=K T_{0}$, and we can take $\widehat{v_{0}}=1$.  Define $V_{0}:=\{v\in V:\, K\cdot \fb_{\hv}\mbox{ is closed }\}$.  
\begin{prop}\label{prop:closed}
The following statements are equivalent. 
\begin{enumerate}
\item $v\in V_{0}$. 
\item For any representative $\hv\in\V$ of $v\in V$, the Borel subalgebra $\fb_{\hv}=\Ad(\hv)\cdot \fb_{0}$ is $\theta$-stable.
\item $\phi(v)=1$.
\item $v\in W^{\theta}\times v_{0}$.
\end{enumerate}
\end{prop}
\begin{proof}

We first show that (1) implies (2).  Let $v\in V_{0}$, and let $B_{\hv}\subset G$ be the Borel subgroup of $G$ corresponding to the Borel subalgebra $\fb_{\hv}$.  Then $K\cdot \fb_{\hv}\subset\B$ is projective, so that the homogeneous space $K/ (K\cap B_{\hv})\cong K\cdot\fb_{\hv}$ is projective, and hence
 $K\cap B_{\hv}$ is parabolic.  Since $K\cap B_{\hv}$ is solvable, it follows
that $K\cap B_{\hv}$ is a Borel subgroup of $K$.  Part (2) now follows from Lemma \ref{l:2closed}.   


We now prove that (2) implies (3).  Suppose that $v\in V$ and that $\fb_{\hv}=\Ad(\hv)\cdot \fb_{0}$ is $\theta$-stable.  Thus, $\Ad(\theta(\hv))\cdot \theta(\fb_{0})=\Ad(\hv)\cdot \fb_{0}$.  But $\fb_{0}$ is itself $\theta$-stable, implying that $\hv^{-1}\theta(\hv)\in B_{0}$.   But then $\hv^{-1}\theta(\hv)=\tau(\hv)\in B_{0}\cap N=T_{0}$ by definition of $\V$.  Thus, $\phi(v)= \tau (\hv)T_0=1$. 

We next show that (3) implies (4).  Suppose that $\phi(v)=1$.  Clearly, $\phi(v_{0})=1$.  It then follows from part (2) of Proposition \ref{prop:actions} that $v=w\times v_{0}$ for some $w\in W$.  But then part (1) of Proposition \ref{prop:actions} implies 
$$
1=\phi(v)=\phi(w\times v_{0})=w\phi(v_{0})\theta(w^{-1})=w\theta(w^{-1}),
$$ 
whence $w\in W^{\theta}$ and $v\in W^{\theta}\times v_{0}$.

Lastly, we show that (4) implies (1).  If $v\in W^{\theta}\times v_{0}$, then $v=K\dot{w} T_{0}$, where $\dot{w}\in N$ is a representative of $w\in W^{\theta}$.  We note that since $w\in W^{\theta}$, $\theta(\dot{w})=\dot{w} t $ for some $t\in T_{0}$.  It follows that $\fb_{\hv}=\Ad(\dot{w})\cdot \fb_{0}$ is $\theta$-stable, since $\ft_{0}\subset\fb_{0}$.  
Let $B_{\hv}$ be the Borel subgroup corresponding to $\fb_{\hv}$, so that $B_{\hv}$ is $\theta$-stable.  It follows from \cite{Rich}, Lemma 5.1 that $B_{\hv} \cap K$ is connected and therefore is a Borel subgroup by Lemma \ref{l:2closed}.  Since $(B_{\hv}\cap K)$ is a Borel subgroup,
 the variety $K/ (B_{\hv}\cap K)$ is complete, and 
 the orbit $K\cdot\fb_{\hv}\cong K/(B_{\hv}\cap K)$ is a complete subvariety of $\B$ and is therefore closed.    



\end{proof}

We now prove Theorem \ref{thm:closed}.  

\begin{proof}[Proof of Theorem \ref{thm:closed}]
It follows from Proposition \ref{prop:closed} that 
\begin{equation}\label{eq:crossorbit}
V_{0}=W^{\theta}\times v_{0}.
\end{equation}
  By \cite{RS}, Proposition 2.8, the stabilizer of $v_{0}$ in $W$ is precisely $W_{K}\subset W^{\theta}$.  Thus, the elements of the orbit $W^{\theta}\times v_{0}$ are in bijection with the coset space $W^{\theta}/ W_{K}$.  Equation (\ref{eq:correspond}) then follows from the definition of the cross action of $V$ on $W$. 
\end{proof}






Recall the notion of the type of a root $\alpha\in\Phi(\fg,\ft_{0})$ for $v$ from Definition \ref{defn:rootsforv}, and note that by Equation (\ref{eq:newinv}), 
\begin{equation}\label{eq:thetav}
\theta_{\hv}=\Ad(\hv^{-1}\theta(\hv))\circ\theta=\Ad(\tau(\hv))\circ\theta.
\end{equation}

\begin{prop}\label{prop:closedinv}
For $v\in V_{0}$, every positive root $\alpha\in\Phi^{+}(\fg,\ft_{0})$ is imaginary or complex $\theta$-stable for $v$.  Moreover, a positive root
 $\alpha\in\Phi^{+}(\fg,\ft_{0})$ is imaginary (resp. complex) for $v$ 
if and only if it is imaginary (resp. complex) for $v_{0}$. 
\end{prop}

\begin{proof}
By Equation (\ref{eq:thetav}), for $v\in V$, 
$\theta_{\hv}(\alpha) = \phi(v) (\theta(\alpha))$
for $\alpha \in \Phi(\fg,\ft_{0})$.  
  Since $v\in V_{0}$, then $\phi(v)=1$ by Proposition \ref{prop:closed}, so 
  \begin{equation}\label{eq:thetasequal}
 \theta_{\hv}(\alpha)=\theta(\alpha)
 \end{equation}
for any $\alpha\in\Phi(\fg,\ft_{0})$.  
Since $\fb_{0}\subset\fg$ is $\theta$-stable, Remark \ref{rem:thetastableroots} implies that any $\alpha\in\Phi^{+}(\fg,\ft_{0})$ is complex $\theta$-stable or imaginary with respect to $\theta$.  Both statements of the proposition then follow immediately from Equation (\ref{eq:thetasequal}).    
\end{proof}

\begin{rem}\label{r:closed}
Let $v\in V_{0}$ and let $\theta_{\hv}$ be the involution associated to the orbit $v$.  To determine the action of $\theta_{\hv}$ on $\Phi(\fg,\ft_{0})$, Proposition \ref{prop:closedinv} implies that it suffices to find which roots are compact (resp. non-compact) imaginary for $v$.  By Theorem \ref{thm:closed}, we may take $\hv=\dot{w^{-1}}$, where $\dot{w^{-1}}$ is a representative for $w^{-1}\in W^{\theta}$.  By Proposition \ref{prop:rootsforv}, it follows that a root $\alpha\in\Phi(\fg,\ft_{0})$ is compact (resp. non-compact) imaginary for $v$ if and only if $w^{-1}(\alpha)$ is compact (resp. non-compact) for the pair $(\Ad(w^{-1})\fb_{0},\ft_{0})$ with respect to $\theta$.
\end{rem}

   \begin{nota} \label{nota:standard}
   We will make use of the following notation for flags in $\C^{n}$.  Let 
   $$
  \mathcal{F}=( V_{0}=\{0\}\subset V_{1}\subset\dots\subset V_{i}\subset\dots\subset V_{n}=\C^{n}).
   $$
   be a flag in $\C^{n}$, with $\dim V_{i}=i$ and $V_{i}=\mbox{span}\{v_{1},\dots, v_{i}\}$, with each $v_{j}\in\C^{n}$.  We will denote this flag $\mathcal{F}$ 
by
   $$
  \mathcal{F}=  (v_{1}\subset  v_{2}\subset\dots\subset v_{i}\subset v_{i+1}\subset\dots\subset v_{n}). 
   $$
   We denote the standard ordered basis of $\C^{n}$ by $\{e_{1},\dots, e_{n}\}$. For $1 \le i, j \le n$, let $E_{ij}$ be the matrix with $1$
in the $(i,j)$-entry and $0$ elsewhere. 
\end{nota}

\begin{exam}\label{ex:closed}
 Let $G=GL(n,\C)$ and let $\theta$ be conjugation by the diagonal matrix $c=diag[1,1, \dots, 1, -1]$.  Then $K=GL(n-1,\C)\times G(1,\C)$ and $\fk=\fgl(n-1,\C)\oplus\fgl(1,\C)$.  Since this involution is inner, $W^{\theta}=W=\mathcal{S}_{n}$, the symmetric group on $n$ letters and $W_{K}=\mathcal{S}_{n-1}$.  We can take $\fb_{0}$ to be the standard Borel subalgebra of $n\times n$ upper triangular matrices and $\ft_{0}\subset\fb_{0}$ to be the diagonal matrices.  By Theorem \ref{thm:closed}, the $n$ closed orbits are then parameterized by the identity permutation and the $n-1$ cycles $\{(n-1\, n),\,(n-2\, n-1\, n),\, \dots, \, (i\dots n),\, \dots , (1\dots n)\}$.  We consider the closed $K$-orbit $v\in V_{0}$ corresponding to the cycle $w=(i\dots n)$.  By Equation (\ref{eq:correspond}), it is generated by the Borel subalgebra $\fb_i := \Ad(w^{-1})\fb_{0}$, which is the stabilizer of the flag:
 \begin{equation}\label{eq:flagi}
{\mathcal{F}}_{i}:= (e_{1}\subset \dots\subset e_{i-1}\subset\underbrace{e_{n}}_{i}\subset e_{i}\subset \dots\subset e_{n-1}).
\end{equation}
Notice that ${\mathcal F}_{n}$ is the standard flag in $\C^{n}$ and 
${\mathcal F}_{1}$ is $K$-conjugate to the opposite flag. 
We denote  $Q_i := K\cdot \fb_i$, so $Q_1, \dots, Q_n$ are the $n$ closed orbits.

Let $\epsilon_{i}\in \ft_{0}^{*}$ be the linear functional $\epsilon_{i}(t)=t_{i}$ for $t\in\ft_{0}$, where $t=\mbox{diag}[t_{1},\dots, t_{i},\dots, t_{n}], \, t_{i}\in\C$.  According to \cite{MO}, any root of the form $\epsilon_{i}-\epsilon_{k}$ or $\epsilon_{k}-\epsilon_{i}$ is non-compact imaginary for $v$ while all other roots are compact imaginary, and the involution $\theta_{\hv}$ associated to $v$ acts on the functionals by $\theta_{\hv}(\epsilon_{i})=\epsilon_{i}$ for all $i$.  The second assertion follows easily from Equation (\ref{eq:thetasequal}).   By Remark \ref{r:closed}, $\alpha=\epsilon_{k}-\epsilon_{j}$ is compact (resp non-compact) imaginary for $v$ if and only if $w^{-1}(\alpha)$ is compact (resp. non-compact) imaginary with respect to $\theta$.  The first assertion then follows from the observation that roots of the form $\epsilon_{n}-\epsilon_{k}$ and $\epsilon_{k}-\epsilon_{n}$ are non-compact imaginary with respect to $\theta$ and all other roots are compact imaginary. 


 \end{exam}

\subsection{ General $K$-orbits in $\B$}\label{s:arbitrary}





  
In this section, we compute $\tau(\hv)$ and $\phi(v)$ inductively
based on the closed orbit case in  Section \ref{s:closed}.  We thus
obtain a formula for $\theta_{\hv}$ for any $K$-orbit in $\B$.


For the first step, we take a $K$-orbit $Q$ and a simple root $\alpha$
and construct a $K$-orbit denoted $m(s_\alpha)\cdot Q$ which 
either coincides with $Q$ or contains $Q$ in its closure as a divisor.
  Let $Q=K\cdot\ \fb_{\hv}\subset\B$ for $v\in V$, 
let $\alpha\in\Phi(\fg,\ft_{0})$ be a simple root, and let $\fp_{\alpha}$ be the minimal parabolic subalgebra generated by $\alpha$.  Let $P_{\alpha}$ denote
the corresponding parabolic subgroup, and let
 $\pi_{\alpha}: G/B_{0}\to G/P_{\alpha}$ denote the canonical projection, which
is a $P_{\alpha}/B_0 = {\mathbb P}^1$-bundle.
\begin{dfn-lem}\label{dfnlem:star}  
The variety $\pi_{\alpha}^{-1}\pi_{\alpha}(Q)$ 
 is irreducible and $K$ acts on $\pi_{\alpha}^{-1}\pi_{\alpha}(Q)$ with finitely many orbits.  The unique open $K$-orbit in $\pi_{\alpha}^{-1}\pi_{\alpha}(Q)$ is denoted by $m(s_\alpha)\cdot Q$.
\end{dfn-lem}
\begin{proof}
Note that $\pi_{\alpha}^{-1}\pi_{\alpha}(Q)
= K \hat{v}P_{\alpha}/B_0$, and it follows easily that $\pi_{\alpha}^{-1}\pi_{\alpha}(Q)$ is irreducible, since it is the image of the double coset $KvP_{\alpha}$ under the projection $p:G\to G/B_{0}$.  The variety $K \hat{v}P_{\alpha}/B_0$ is clearly $K$-stable.  It follows that it has a unique open orbit, since the set of $K$-orbits in $K \hat{v}P_{\alpha}/B_0$  is a subset of the set of $K$-orbits on
$\B$, and hence is finite.
\end{proof}

The orbit  $m(s_\alpha)\cdot Q$ may be equal to $Q$ itself.  However, in the case where $m(s_\alpha)\cdot Q\neq Q$, then $\dim m(s_\alpha)\cdot Q=\dim Q+1$, since the map $\pi_{\alpha}: G/B_{0}\to G/P_{\alpha}$ is a $\mathbb{P}^{1}$-bundle.
To compute $m(s_\alpha)\cdot Q$ explicitly (following \cite{Vg}, Lemma 5.1), we recall first some facts about
involutions for $SL(2, \C)$.

Let $\Pi$ denote the set of simple roots with respect to $\ft_{0}$ and let $\alpha\in\Pi$.  Let $h_{\alpha}=\frac{2 H_{\alpha}}{<\alpha,\alpha>}$ with $H_{\alpha}\in\ft_{0}$ such that $<H_{\alpha},x>=\alpha(x)$ for $x\in\ft_{0}$, and let $e_{\alpha}\in\fg_{\alpha}$, $f_{\alpha}\in \fg_{-\alpha}$ be chosen so that $[e_{\alpha}, f_{\alpha}]=h_{\alpha}$.  Hence, the subalgebra $\fs(\alpha)=\mbox{span}\{e_{\alpha}, f_{\alpha}, h_{\alpha}\}$ forms a Lie algebra isomorphic to $\mathfrak{sl}(2,\C)$. 
Let $\phi_{\alpha}: \mathfrak{sl}(2)\to \fs(\alpha)$ be the map
\begin{equation}\label{eq:sl2map}
\phi_{\alpha}:\left[\begin{array}{cc} 0 & 1\\
0 & 0\end{array}\right]\to e_{\alpha},\;    
\phi_{\alpha}:\left[\begin{array}{cc} 0 & 0\\
1 & 0\end{array}\right]\to f_{\alpha},\;
\phi_{\alpha}:\left[\begin{array}{cc} 1 & 0\\
0 & -1\end{array}\right]\to h_{\alpha}\; 
\end{equation}
Then $\phi_{\alpha}:\fsl(2)\to\fs(\alpha)$ is a Lie algebra isomorphism, which integrates to an injective homomorphism of Lie groups $\phi_{\alpha}: SL(2,\C) \to G$, which we will also denote by $\phi_{\alpha}$. We let $S(\alpha)$ be
its image.

To perform computations, it is convenient for us to choose specific representatives for the Cayley transform $u_{\alpha}$ with respect to $\alpha$ and the simple reflection $s_{\alpha}$.  Let 
\begin{equation}\label{eq:cayley}
u_{\alpha}=\phi_{\alpha}\left(\frac{1}{\sqrt{2}}\left[\begin{array}{cc} 1 & \imath\\
\imath & 1\end{array}\right]\right).
\end{equation}
Note that $g=  \frac{1}{\sqrt{2}}\left[\begin{array}{cc} 1 & \imath\\
\imath & 1\end{array}\right]\in SL(2,\C)$ is the Cayley transform which conjugates the torus in $SL(2,\C)$ containing the diagonal split maximal torus of
$SL(2,\R)$ to a torus of $SL(2,\C)$ containing a compact maximal torus 
of $SL(2,\R)$.   Let 
\begin{equation}\label{eq:dotsalpha}
\dot{s}_{\alpha}=\phi_{\alpha}\left(\left[\begin{array}{cc} 0 & \imath\\
\imath & 0\end{array}\right]\right).
\end{equation}
  Then $\dot{s}_{\alpha}$ is a representative for $s_{\alpha}\in W$.  Note that $u_{\alpha}^{2}=\dot{s}_{\alpha}$. 

  Let $\theta_{1,1}:\,SL(2,\C)\to SL(2,\C)$ be the involution on $SL(2,\C)$ given by 
$$
\theta_{1,1}(g)=\left[\begin{array}{cc} 1 & 0\\
0 & -1\end{array}\right] g\left[\begin{array}{cc} 1 & 0\\
0 & -1\end{array}\right] 
$$
for $g\in SL(2,\C)$.   

 \begin{lem}\label{l:imag}
 Suppose $\alpha\in\Pi$ is compact (resp non-compact) imaginary for $v$.  Then $-\alpha$ is compact (resp non-compact) imaginary for $v$.  
 \end{lem}
\begin{proof}
Since $\theta_{\hv}(\fg_{\alpha})=\fg_{\alpha}$, it follows easily that
$\theta_{\hv}(\fg_{-\alpha})=\fg_{-\alpha}.$   The rest of the proof follows since
$\theta_{\hv}$ preserves the Killing form.
\end{proof}

\begin{lem}\label{l:sl2}
If $\alpha$ is non-compact imaginary for $v$, then 
\begin{equation}\label{eq:sl2}
\theta_{\hv}\circ\phi_{\alpha}=\phi_{\alpha}\circ \theta_{1,1}.
\end{equation}
\end{lem}
\begin{proof}
 It suffices to verify Equation (\ref{eq:sl2}) on the Lie algebra $\fsl(2,\C)$.   On $\fsl(2,\C)$ the maps in Equation (\ref{eq:sl2}) are linear, and we need only check the equation on a basis for $\fsl(2,\C)$.  Since $\alpha$ is non-compact imaginary for $v$, we have $\theta_{\hv}(e_{\alpha})=-e_{\alpha}$, $\theta_{\hv}(f_{\alpha})=-f_{\alpha}$, and $\theta_{\hv}(h_{\alpha})=h_{\alpha}$ by Lemma \ref{l:imag}, and the result follows.
\end{proof}

\begin{rem}\label{r:thetavsl2}
It follows from the proof of Lemma \ref{l:sl2} that $\fs(\alpha)^{\theta_{\hv}} = \C h_\alpha$.
\end{rem}

\begin{prop}\label{p:diffroots}
Let $Q=K\cdot\fb_{\hv}$ with $v\in V$ and let $\alpha\in \Phi(\fg, \ft_{0})$ be a simple root.  Then $m(s_\alpha)\cdot Q\neq Q$ if and only if $\alpha$ is non-compact imaginary for $v$ or $\alpha$ is complex $\theta$-stable for $v$.  If $\alpha$ is non-compact imaginary, then $m(s_\alpha)\cdot Q=K\cdot \fb^{\prime}$, with $\fb^{\prime}=\Ad(\hv u_{\alpha}) \fb_{0}$, where $u_{\alpha}$ is the Cayley transform with respect to $\alpha$.  If $\alpha$ is complex $\theta$-stable, then $m(s_{\alpha})\cdot Q=K\cdot\fb^{\prime}$, with $\fb^{\prime}=\Ad(\hv s_{\alpha})\fb_{0}$. 
\end{prop}

\begin{proof}
Let $K_{\hv} = K \cap \Ad(\hv) P_{\alpha}$ be the stabilizer in
$K$ of $\pi_{\alpha}({\hv}B_0/B_{0})$. 
Let $L_{\hv}=\pi_{\alpha}^{-1}\pi_{\alpha}(\hv B_0/B_0)$, which is identified
with $\Ad(\hv)P_\alpha/\Ad(\hv)B_0 \cong \mathbb{P}^{1}$.
We claim that the map $\chi$ from the set of $K_{\hv}$-orbits in $L_{\hv}$
to the set of $K$-orbits in $K \hv P_\alpha /B_0$ given by
$\chi(\hQ)=K\cdot \hQ$ is bijective.
Indeed, if $Q_1 \subset K \hv P_{\alpha} /B_0$ is a $K$-orbit, then for $z_1, z_2
\in Q_1 \cap L_{\hv}$, we have $z_2 = k\cdot z_1$ for some $k\in K$,
and $\pi_{\alpha}(z_1)=\pi_{\alpha}(z_2)$.  It follows that $k$ stabilizes
$\pi_{\alpha}({\hv}B_0/B_{0})$, so $k\in K_{\hv}$.  Hence, $Q_1 \cap L_{\hv}$
is a $K_{\hv}$-orbit, and it is routine to check that $Q_1 \mapsto
Q_1 \cap L_{\hv}$ is inverse to $\chi$, giving the claim.  
Let $U^{\alpha}$ be the unipotent radical of $P_{\alpha}$, and let
$Z(M_{\alpha})^0$ be the identity component of the center of a Levi subgroup
of $P_{\alpha}$.   Then $\Ad(\hv)P_\alpha$ acts on the fiber $L_{\hv}$ through 
its quotient $\tilde{S}_{\hv}:=\Ad(\hv)P_{\alpha}/\Ad(\hv)(Z(M_{\alpha})^0 U^{\alpha})$,
which is locally isomorphic to $\Ad(\hv)S(\alpha)$.  Hence $K_{\hv}$ 
acts on $L_{\hv}$ through its image ${\tilde{K}}_{\hv}$ in $\tilde{S}_{\hv}$.
For $\alpha$  non-compact imaginary for $v$,
it follows from Remark \ref{r:thetavsl2} that  ${\tilde{K}}_{\hv}$ has
Lie algebra $\Ad(\hv)(\C h_{\alpha})$, and hence
${\tilde{K}}_{\hv}$  is either a torus of $\tilde{S}_{\hv}$ normalizing $\hv  B_{0}/B_{0}$ or the normalizer of such a torus.  Hence, the points $\hv B_0/B_0$ and $\hv s_\alpha B_0/B_0$ are
in zero-dimensional ${\tilde{K}}_{\hv}$-orbits, and the complement
$L_{\hv} - (\hv B_0/B_0 \cup \hv s_\alpha B_0/B_0)$ is a single 
${\tilde{K}}_{\hv}$-orbit containing $\hv u_\alpha B_0/B_0$.  From the
definition of the bijection $\chi$, it follows
that $K\hv B_0/B_0$ is a proper subset of $\overline{K\hv u_\alpha B_0/B_0}$, where the closure is taken in the variety $K\hv P_{\alpha}/B_{0}$.
Since $\dim(K\hv P_{\alpha}/B_0)=\dim(K\hv B_0/B_0) + 1$, we conclude
that $m(s_\alpha)\cdot Q = K\hv u_{\alpha} B_0/B_0$.
This verifies the proposition in the case of non-compact
imaginary roots, and the other cases are similar, and discussed
in detail in section 2 of \cite{RSexp}.
\end{proof} 

\begin{rem}
In \cite{Vg}, the author discriminates between two types of non-compact roots.
For $G=GL(n,\C)$ and $K=GL(p,\C)\times GL(n-p,\C)$, all non-compact roots
for all orbits are type I.  
\end{rem}




\begin{nota}\label{nota:qijdef}
We let $G=GL(n,\C)$ and $K=GL(n-1,\C)\times G(1,\C)$ as in Example
\ref{ex:closed}.  We let $\fb_{i,j}$ be the Borel subalgebra stabilizing
the flag
$$
 {\mathcal{F}}_{i,j} =   (e_{1}\subset \dots\subset \underbrace{e_{i}+e_{n}}_{i}\subset e_{i+1}\subset \dots \subset e_{j-1} \subset \underbrace{e_{i}}_{j}\subset e_{j} \subset \dots \subset e_{n-1}),
$$
and we let $Q_{i,j}=K\cdot \fb_{i,j}$.
\end{nota}

\begin{exam}\label{ex:noncptimag}
We let $G$ and $K$ be as in Example \ref{ex:closed} and compute
$m(s_\alpha)\cdot Q_c$ for each closed $K$-orbit $Q_c$.  By Example
\ref{ex:closed}, $Q_{c}= Q_i =K\cdot\fb_{i}$, where $\fb_{i}$ is the stabilizer of the flag ${\mathcal{F}}_i$
from Equation \eqref{eq:flagi}.  Let $Q_{i}:=K\cdot\fb_{i}$ and let $v_i$ be the corresponding element
of $V$.  By Example \ref{ex:closed}, the simple
roots $\alpha_{i-1}= \epsilon_{i-1} - \epsilon_{i}$ and $\alpha_i =
\epsilon_{i} - \epsilon_{i+1}$ are the only non-compact imaginary
simple roots for $Q_i$, and all other simple roots are compact (for $i=1$
and $i=n$, one of these two roots does not exist).  Since
$Q_i = K \cdot \dot{w} \fb_0$, where $\dot{w}$ is a representative
for the element $(n\, \dots\, i)$ of $W$, it follows from 
Proposition \ref{p:diffroots} that $m(s_{\alpha_{i-1}})\cdot Q_i = 
K\cdot \dot{w}u_{\alpha_{i-1}} \fb_0$.  A routine computation
shows that the K-orbit $K\cdot \dot{w}u_{\alpha_{i-1}} \fb_0$
contains the stabilizer of the flag
$$
 {\mathcal{F}}_{i-1,i} =   (e_{1}\subset \dots\subset \underbrace{e_{i-1}+e_{n}}_{i-1}\subset e_{i-1}\subset \dots  \subset e_{n-1}).
$$
Hence, 
\begin{equation}\label{eq:monoidclosedminus}
m(s_{\alpha_{i-1}})\cdot Q_i = Q_{i-1,i}.
\end{equation}
A similar calculation shows that
\begin{equation}\label{eq:monoidclosed}
m(s_{\alpha_{i}})\cdot Q_i = Q_{i,i+1}.
\end{equation}
\end{exam}

Let $Q_{c}=K\cdot\fb_{\hv}$ be a closed $K$-orbit and let $B_{\hv}\subset G$ be the Borel subgroup with $Lie(B_{\hv})=\fb_{\hv}$.  We observed in the proof of Proposition \ref{prop:closed} that $K\cap B_{\hv}$ is a Borel subgroup of $K$ so that $Q_{c}\cong K/ (K\cap B_{\hv})$ is isomorphic to the flag variety $\B_K$ of $K$.  


\begin{dfn-nota}
 For a $K$-orbit $Q$ on $\B$, we let $l(Q):= \dim(Q) - \dim(\B_K)$.  The number $l(Q)$ is called the \emph{length} of the $K$-orbit $Q$. 
\end{dfn-nota}

\begin{prop}\label{prop:chain}
Let $Q$ be any $K$-orbit in $\mathcal{B}$.  Then there exists a sequence of simple roots $\alpha_{i_{1}},\cdots,\alpha_{i_{k}}\in \Phi^{+}(\fg,\ft_{0})$ and a closed orbit $Q_{c}$ such that $Q=m(s_{\alpha_{i_{k}}})\cdot\ldots \cdot m(s_{\alpha_{i_{1}}})\cdot Q_{c}$.  We let $Q_j=m(s_{\alpha_{i_{j}}})\cdot\ldots \cdot m(s_{\alpha_{i_{1}}})\cdot Q_{c}$.   If
for $j=1, \dots, k$, the root $\alpha_{i_{j}}$ is complex $\theta$-stable 
or non-compact imaginary for $Q_{j-1}$, then $l(Q)=k$.
\end{prop}
\begin{proof}
This follows easily from \cite{RS}, Theorem 4.6 .
\end{proof}


Let $Q_{v}$ be the $K$-orbit corresponding to $v\in V$.  We now compute 
the involution associated to the orbit $m(s_\alpha)\cdot Q_{v}$ when $\alpha$ is complex $\theta$-stable or non-compact imaginary for $v$ from the involution for the orbit $Q_{v}$.  We denote the parameter $v^{\prime}\in V$ for $m(s_\alpha)\cdot Q_{v}$ by $v^{\prime}=m(s_{\alpha})\cdot v$.  By results from Section \ref{s:closed} and Proposition \ref{prop:chain}, we can then determine $\theta_{\widehat{v^{\prime}}}$ for any $v^{\prime}$ in $V$.

There are two different cases we need to consider.  

\noindent {\bf Case 1}: $\alpha$ is non-compact imaginary for $v$.  Let $v^{\prime}=m(s_\alpha)\cdot v$.  Then by Proposition \ref{p:diffroots}, $K\cdot \fb_{\widehat{v^{\prime}}}=K\cdot \Ad(\hv u_{\alpha})\fb_{0}$, where $u_{\alpha}$ is the representative for the Cayley transform with respect to $\alpha$ given in Equation (\ref{eq:cayley}).  
  
%
%

We can now compute $\theta_{\widehat{v^{\prime}}}$ in terms of $\theta_{\widehat{v}}$.  

\begin{prop}\label{p:noncpt}
Let $v^{\prime}=m(s_\alpha)\cdot v$, where $\alpha$ is non-compact imaginary for $v$.  
\par\noindent (1) Then $\hv u_{\alpha}\in\V$ is a representative of $v^{\prime}$, and 
$$\tau(\widehat{v^{\prime}})= \tau(\hv u_{\alpha})=\dot{s_{\alpha}}^{-1} \tau(\hat{v}),$$
and 
$$ \phi(v^{\prime})=s_{\alpha} \phi(v).$$  
\par\noindent (2) 
The involution for $v^{\prime}$ is given by 
$$
\theta_{\widehat{v^{\prime}}}= \Ad(\tau(\widehat{v^{\prime}}))\circ \theta=\Ad(\dot{s_{\alpha}}^{-1})\Ad(\tau(\hat{v}))\circ \theta=\Ad(\dot{s_{\alpha}}^{-1})\circ \theta_{\hv},
$$
and $\theta_{\widehat{v^{\prime}}}$ acts on the roots $\Phi(\fg,\ft_{0})$ by:
$$\theta_{\widehat{v^{\prime}}}= s_{\alpha} \theta_{\hat{v}}.$$
\end{prop}

\begin{proof}
It is easy to verify that if 
$g=\frac{1}{\sqrt{2}}\left[\begin{array}{cc} 1 & \imath\\
\imath & 1\end{array}\right]$, then $\theta_{1,1}(g)=g^{-1}$.
Hence, by Lemma \ref{l:sl2}, it follows that
$\theta_{\hv}(u_{\alpha})=u_{\alpha}^{-1}$.  Thus, by Equation
\eqref{eq:thetav}, $\theta(u_{\alpha})=\Ad(\tau(\hv)^{-1})(u_{\alpha}^{-1}).$
It follows that
$$
\tau(\hv u_\alpha) = u_{\alpha}^{-1} \tau(\hv) \theta(u_{\alpha})
   = u_{\alpha}^{-1} \tau(\hv) \tau(\hv)^{-1} u_{\alpha}^{-1} \tau(\hv)
   = u_{\alpha}^{-2} \tau(\hv).$$
Since $u_{\alpha}^{-2}= \dot{s_{\alpha}}^{-1}$, it follows that $\tau(\hv u_\alpha)
= \dot{s_{\alpha}}^{-1} \tau(\hv)$.  By Equation (\ref{eq:varietyV}) and Proposition \ref{p:diffroots}, it follows
that $\hv u_{\alpha}\in \mathcal{V}$ is a representative of $m(s_\alpha)\cdot v$.
  By Equation \eqref{eq:mapphi}, we have $\phi(m(s_{\alpha})\cdot v)=s_{\alpha}\phi(v)$.  Part (2) of the proposition now follows from part (1) and Equation (\ref{eq:thetav}). 



\end{proof}

\noindent{\bf Case 2}:  $\alpha$ is complex $\theta$-stable for $v$.



\begin{prop}\label{p:cplx}
Let $\alpha$ be complex $\theta$-stable for $v$.  
\par\noindent (1) Let $v^{\prime}=m(s_{\alpha})\cdot v$.  Then $v^{\prime}$ has representative $\widehat{v^{\prime}}=\hv \dot{s_{\alpha}}$, so that $ v^{\prime}=s_{\alpha}\times v\in V$ 
and $$ \tau(\hv\dot{s_{\alpha}})=\dot{s_{\alpha}}^{-1}\tau(\hv)\theta(\dot{s_{\alpha}}),$$ whence
$$\phi(v^{\prime})=s_{\alpha}\phi(v)\theta(s_{\alpha}).$$ 
\par\noindent (2) The involution $\theta_{\widehat{v^{\prime}}}$ on $\fg$ associated to $v^{\prime}$ is given by 
$$
\theta_{\widehat{v^{\prime}}}=\Ad(\dot{s_{\alpha}}^{-1}\tau(\hv)\theta(\dot{s_{\alpha}}))\circ \theta=\Ad(\dot{s_{\alpha}}^{-1})\circ \theta_{\hv}\circ \Ad(\dot{s_{\alpha}}), 
$$
so that the action of $\theta_{\widehat{v^{\prime}}}$ on the roots $\Phi(\fg, \ft_{0})$ is given by: 
 $$\theta_{\widehat{v^{\prime}}}=s_{\alpha}\phi(v)\theta(s_{\alpha})\theta = s_\alpha \theta_{\hv} s_{\alpha}.$$
\end{prop}
\begin{proof}
By Proposition \ref{p:diffroots}, we have 
$ \fb_{\widehat{v^{\prime}}}=\Ad(\hv \dot{s_{\alpha}})\cdot \fb_{0}$ so that $v^{\prime}=s_{\alpha}\times v$ by Equation \eqref{eq:crossaction}.  The rest of the proof
follows by definitions. 
\end{proof}

\begin{lem}\label{l:noncptnextorbit}
Let $Q_v$ be the $K$-orbit corresponding to $v\in V$, and let $\alpha$ be a complex $\theta$-stable
simple root for $v$.  Let $\beta$ be a root of $\Phi^+(\fg, \ft_0)$.
Then $\beta$ is noncompact imaginary for $v$ if and only if 
$s_{\alpha}(\beta)$ is non-compact imaginary for $m(s_{\alpha})\cdot v$.
\end{lem}

\begin{proof}
Let $v^{\prime}=m(s_{\alpha})\cdot v$.  Then by Proposition \ref{p:cplx} (2),
$\theta_{\widehat{v^{\prime}}}(s_{\alpha}(\beta)) = s_{\alpha}(\theta_{\hat{v}}(\beta))$.
Hence, $\beta$ is imaginary for $v$ if and only if $s_{\alpha}(\beta)$
is imaginary for $v^{\prime}$.  To prove the non-compactness assertion,
it suffices to apply Proposition \ref{p:cplx} (2) to a root vector
$\Ad(\dot{s_{\alpha}}^{-1})(x_{\beta})$, where $x_{\beta}$ is a nonzero root
vector in $\fg_{\beta}$.
\end{proof}

\begin{exam}\label{ex:qijthetaroots}
We show how this theory helps describe the $K$-orbits
 $Q_{i,j}$ in the case when $G=GL(n,\C)$ and $K=GL(n-1,\C)\times G(1,\C)$.
We let $v_{i,i+1} \in V$ parametrize the orbit $Q_{i,i+1}$.  By Equation \eqref{eq:monoidclosed} and Propositions \ref{prop:closed} and \ref{p:noncpt} (1),  
the Springer invariant $\phi(v_{i,i+1})=(i \ i+1)=s_{\alpha_{i}}$, and using
also Example \ref{ex:closed}, $v_{i,i+1}$ has representative $\widehat{v_{i,i+1}}=(n\, n-1\dots i) u_{\alpha_{i}}$, where $u_{\alpha_{i}}$ is the Cayley transform from Equation (\ref{eq:cayley}).  Hence,
$\alpha_i$ is real for $v_{i,i+1}$, while $\alpha_{i-1}$ and $\alpha_{i+1}$
are the only $\theta$-stable complex simple roots (as before, in case
$i=1$ or $n-1$, only one of these complex roots exists).
Further, the imaginary roots for $v_{i,i+1}$ are the roots 
$\epsilon_{j} - \epsilon_{k}$
with $j, k \not\in \{ i, i+1 \}$ and have root vectors $E_{jk}$.  
Then by Proposition \ref{p:noncpt} (2),
$\theta_{\widehat{v_{i,i+1}}}(E_{jk})=\Ad(\dot{s_{\alpha_i}}^{-1})\theta_{\widehat{v_i}}(E_{jk})$, where $\dot{s_{\alpha_{i}}}$ is the representative for $s_{\alpha_{i}}\in W$ given in Equation (\ref{eq:dotsalpha}).  But by Example \ref{ex:closed}, $\theta_{\widehat{v_{i}}}(E_{jk})=E_{jk}$, so the roots $\epsilon_{j} - \epsilon_{k}$ are compact.
Hence, there are no non-compact imaginary roots for $Q_{i,i+1}$.

We now consider all orbits $Q_{i,j}$ with $i < j$.  We let $v_{i,j}\in V$ denote the corresponding parameter, and we let
$s_i = (i \ i+1)$ with representative $\dot{s_i}$ given by the corresponding permutation matrix.  

\noindent{\bf Claim:} 
\begin{enumerate}
\item $Q_{i,j}=m(s_{j-1})\cdot\ldots\cdot m(s_{i})\cdot Q_{i}$ and $l(Q_{i,j})=j-i$.

\item $\phi(v_{i,j})$ is the transposition $(i\, j)$, $\theta_{\widehat{v_{i,j}}}=(i\,j)$
on roots, and $Q_{i,j}$ has
representative given by the element $\widehat{v_{i,j}}=(n\, n-1\,\dots\, i)u_{\alpha_{i}}{\dot{s}}_{i+1} \dots {\dot{s}}_{j-1}$.

\item The simple roots $\alpha_{i-1} = \epsilon_{i-1} - \epsilon_{i}$
and $\alpha_j=\epsilon_{j} - \epsilon_{j+1}$ are the only complex $\theta$-stable simple roots for $v_{i,j}$, and there are no non-compact imaginary roots for $v_{i,j}$.

\end{enumerate}
We prove these claims by induction on $j-i$.  Example \ref{ex:noncptimag} and our discussion in the first paragraph proves the claim when $j-i=1$.  It suffices to show that (1)-(3) of the claim for $Q_{i,j}$
imply the claim for $Q_{i,j+1}$.  By Proposition \ref{p:diffroots} and Claims (2) and (3) for $Q_{i,j}$, it follows that $m(s_{j})\cdot Q_{i,j}\neq Q_{i,j}$ and $m(s_{j})\cdot Q_{i,j}$ has
representative $\widehat{v_{i,j+1}}$.  A routine computation with flags then shows that 
$K\widehat{v_{i,j+1}}\fb_{0}=Q_{i,j+1}$.  Hence,
\begin{equation}\label{eq:mi1qij}
m(s_{j})\cdot Q_{i,j}=Q_{i,j+1}.
\end{equation}
Claim (1) for $Q_{i,j+1}$ then follows by induction.  Claim (2) for $Q_{i,j}$ and Proposition \ref{p:cplx} (1) imply that $\phi(v_{i, j+1})$ is the transposition $(i\, j+1)$.  The formula for $\theta_{\widehat{v_{i,j}}}$ in Claim (2) follows
from Proposition \ref{p:cplx} part (2).  Claim (3) now follows by Lemma \ref{l:noncptnextorbit} and an easy computation.  This verifies Claims (1)-(3) for the orbit $Q_{i,j+1}$.  

We remark that a computation similar to the one above
 verifies that 
\begin{equation}\label{eq:miqij}
m(s_{i-1})\cdot Q_{i,j}=Q_{i-1,j}.
\end{equation}


\end{exam}



\begin{exam}\label{ex:allorbitsqiqij}
We retain the notation from the last example.
We assert that every $K$-orbit $Q$ in $\B$ is either of the
form $Q_i$ or $Q_{i,j}$ with $i < j$ and that these orbits are all distinct.  We prove the first assertion
by induction on $l(Q)$.  If $l(Q)=0$, then $Q$ is closed,
so $Q=Q_i$ by Example \ref{ex:closed}.  If $l(Q)=1$, then
by Proposition \ref{prop:chain}, $Q=m(s_i)\cdot Q_c$ for some
closed orbit $Q_c$, so by Example \ref{ex:noncptimag} and
Equations \eqref{eq:monoidclosedminus} and \eqref{eq:monoidclosed},
 it follows that $Q=Q_{i,i+1}$ for some $i$.
If $l(Q)=k > 1$, then Proposition \ref{prop:chain} implies $Q=m(s_{i})\cdot \tilde{Q}$, where $l(\tilde{Q})=k-1$,
so by induction $\tilde{Q} = Q_{j,j+k-1}$ for some $j$, and by Claim (3)
of Example \ref{ex:qijthetaroots}, the simple root $\alpha_{i}$
is either $\alpha_{j-1}$ or $\alpha_{j+k-1}$.   The first assertion
now follows by Equations \eqref{eq:mi1qij} and \eqref{eq:miqij}.  By Example \ref{ex:closed},  the orbits $Q_{i}$ 
are distinct.  By Claim (2) of Example \ref{ex:qijthetaroots}, the Springer 
invariant for $Q_{i,j}$ is $(i\, j)$, so that $Q_{i,j}=Q_{i^{\prime},j^{\prime}}$ if and only if $i=i^{\prime}$ and $j=j^{\prime}$.  We now 
have a complete classification of the $K$-orbits on $\mathcal{B}$.  
\end{exam}

\begin{exam}\label{ex:openorbit} 
We claim that $Q_{1,n}$ is the unique open orbit of $K$ on $\B$,
where we retain notation from the previous two examples.
Indeed, by Claim (1) from Example \ref{ex:qijthetaroots},
$l(Q_{1,n})=n-1=\dim Q_{1,n} -\dim(\B_K)$, so that $\dim Q_{1,n} = n-1+\dim(\B_{K})=\dim(\B)$.  It follows
that $Q_{1,n}$ is open in $\B$.
\end{exam}

\begin{rem}\label{r:selecta}
The last three examples verify the assertions of \cite{Yam}, Section 2 and \cite{MO} for the
case $G=GL(n,\C)$ and $K=GL(n-1,\C) \times GL(1,\C)$.  In particular, they justify the 
statements made in \cite{CEKorbs}, Section 3.1.  Example \ref{ex:qijthetaroots} explains the definition of the element $v$ in Equation (3.3) and the construction of the involution $\theta^{\prime}$ in \cite{CEKorbs}, Section 3.1, which is the critical ingredient in the proof of Theorem \ref{thm:strong} above (see Remark \ref{rem:bigproof}).  
\end{rem}



\bibliographystyle{amsalpha.bst}

\bibliography{bibliography}








  

\end{document}